\documentclass[12pt]{amsart}
\usepackage{amsfonts}
\usepackage{amsmath}
\usepackage{amssymb}
\usepackage{url,epsfig}
\usepackage{color}
\usepackage{graphicx}
\usepackage{xcolor}
\usepackage{ulem}
\usepackage{cancel}

\newcommand{\cred}[1]{{\color{black} #1}}

\definecolor{rot}{rgb}{0,0,0}
\newcommand{\tcr}{\textcolor{rot}}

\def\eps{\varepsilon}

\newcommand\cC{\mathcal{C}}

\newcommand\cK{\mathcal{K}}

\newcommand\cM{\mathcal{M}}
\newcommand\cN{\mathcal{N}}

\newcommand\cR{\mathcal{R}}
\newcommand\cT{\mathcal{T}}

\newcommand\bE{\boldsymbol{E}}
\newcommand\bF{\boldsymbol{F}}
\newcommand\bH{\boldsymbol{H}}

\newcommand\bn{\boldsymbol{n}}
\newcommand\bp{\boldsymbol{p}}
\newcommand\bq{\boldsymbol{q}}
\newcommand\bs{\boldsymbol{s}}
\newcommand\bu{\boldsymbol{u}}
\newcommand\bv{\boldsymbol{v}}

\newcommand\bff{\boldsymbol{f}}
\newcommand\btheta{\boldsymbol{\theta}}
\newcommand\bnu{\boldsymbol{\nu}}
\newcommand\bpsi{\boldsymbol{\psi}}
\newcommand\balpha{\boldsymbol{\alpha}}

\newcommand{\ddiv}{\operatorname{div}}
\newcommand{\curl}{\operatorname{curl}}
\newcommand{\Div}{\operatorname{Div}}
\newcommand{\Curl}{\operatorname{Curl}}
\newcommand{\Grad}{\operatorname{Grad}}

\renewcommand{\Re}{\operatorname{Re}}
\renewcommand{\Im}{\operatorname{Im}}

\newcommand{\smalf}{\par\smallskip\noindent}
\newcommand{\medlf}{\par\medskip\noindent}

\newcommand{\per}{\mathrm{per}}
\newcommand{\loc}{\mathrm{loc}}

\newcommand{\be}{\begin{eqnarray}}
\newcommand{\ben}{\begin{eqnarray*}}
\newcommand{\en}{\end{eqnarray}}
\newcommand{\enn}{\end{eqnarray*}}
\newcommand{\Z}{{\mathbb Z}}
\newcommand{\N}{{\mathbb N}}
\newcommand{\C}{{\mathbb C}}
\newcommand{\R}{{\mathbb R}}
\newcommand{\s}{{\mathbb S}}

\newtheorem{theorem}{Theorem}[section]
\newtheorem{definition}[theorem]{Definition}
\newtheorem{lemma}[theorem]{Lemma}

\newtheorem{remark}[theorem]{Remark}

\newtheorem{example}[theorem]{Example}
\newtheorem{assumption}[theorem]{Assumption}

\begin{document}

\title{Limiting absorption principle for time-harmonic acoustic and
electromagnetic scattering of plane waves from a bi-periodic inhomogeneous layer}

\author{Guanghui Hu}
\address{Guanghui Hu: School of Mathematical Sciences and LPMC\\
Nankai University \\
Tianjin 300071, China}
\email{ghhu@nankai.edu.cn}
\author{Andreas Kirsch}
\address{Andreas Kirsch: Department of Mathematics \\
Karlsruhe Institute of Technology (KIT) \\
76131 Karlsruhe, Germany}
\email{andreas.kirsch@kit.edu}
\author{Yulong Zhong}
\address{Yulong Zhong: School of Mathematical Sciences and LPMC\\
	Nankai University \\
	Tianjin 300071, China}
\email{zyl99@mail.nankai.edu.cn}

\date{\today \\ The first author (G.H.) was partially supported by the National Natural Science Foundation of China (No. 12071236), the Fundamental Research Funds for Central Universities in China (No. 63213025) and the Natural Science Foundation of Tianjin (No. 25JCZDJC00970). He greatly acknowledges the hospitality of the Institute for Applied and Numerical Mathematics, Karlsruhe Institute of Technology as well as the Alexander von Humboldt-Stiftung.
The second author (A.K.) gratefully
acknowledge the financial support by Deutsche Forschungsgemeinschaft (DFG) through
CRC\,1173}

\begin{abstract}
The Rayleigh expansion is widely used as a formal radiation condition in the analysis
and numerical treatment of grating diffraction problems for incoming plane waves.
However, the Rayleigh expansion does not always lead to uniqueness of open waveguide
scattering problems, due to the existence of surface/guided waves (in other words,
Bound States in the Continuum (BICs)) which exponentially decay in the direction
perpendicular to the periodicity. In this paper we suppose that  a bi-periodic
inhomogeneous medium supports BICs at some real-valued incident wavenumber. Based
on singular perturbation arguments, we justify the Limiting Absorption Principle
(LAP) for both time-harmonic acoustic and electromagnetic scattering of plane waves
from bi-periodic structures. Replacing the wavenumber $k$ with $k+i\epsilon$, we
prove that the unique solution with $\epsilon>0$ converges to a solution of the
original diffraction problem that additionally satisfies an orthogonal identity. This
constraint condition together with the classical Rayleigh expansion leads to a sharp
radiation condition to ensure uniqueness of time-harmonic scattering of plane waves
by  BIC-supporting bi-periodic materials.

\vspace{.2in}

\noindent
Keywords: Plane wave, limiting absorption principle, diffraction gratings,
Helmholtz equation, Maxwell system, uniqueness.
\end{abstract}

\maketitle

%Version 8 of \today

\section{Introduction}

Time-harmonic scattering problems for layered structures require appropriate
radiation conditions to ensure solution uniqueness \cite{Weder, Xu}. The Rayleigh
expansion is widely employed in analyzing and numerically solving  grating
diffraction problems for incident plane waves \cite{AT,K93, R07, P1980}. However, the
Rayleigh expansion does not always guarantee uniqueness \cite{BBS94, KN02, G00}. This
is due to the potential existence of surface or guided waves -specifically, Bound
States in the Continuum (BICs)- which decay exponentially along the direction
perpendicular to the direction of periodicity. Resonance occurs only when the wave
field's quasi-periodic momentum, uniquely determined by the incident wavenumber and
angles, coincides with a  {propagative wave vector} (that is, when the homogeneous problem
admits non-trivial propagating solutions under the classical Rayleigh expansion
radiation condition). It has been shown in the literatures that for a fixed incident
angle, the scattering problem is well-posed except for a discrete set of incident
wavenumbers \cite{AT,B97,GP22,D93,ES98}. This work aims to derive a novel radiation
condition to ensure uniqueness at these exceptional wavenumbers.
\smalf
The Limiting Absorption Principle (LAP) is a fundamental concept in wave scattering
theory that resolves resonance singularities by introducing artificial dissipation,
ensuring the well-posedness of the mathematical model; see e.g., \cite{Hoang,Nosich,
Weder} for its application in layered structures and waveguides. Recently, the LAP
has been employed to derive radiation conditions for scattering problems involving
compactly supported source terms in both closed and open periodic waveguides
\cite{FJ16, K22p, K22, KL18}. This approach fundamentally relies on the one-dimensional
Floquet-Bloch transform and singular perturbation arguments involving two variables.
However, this methodology is currently restricted to periodic media varying in only
one direction. For bi-periodic structures, the complexity of periodic Helmholtz
equation's spectral structure (dispersion map) introduces essential difficulties in
handling compactly supported source terms in $\R^3$.
\smalf
Time-harmonic scattering of plane waves by bi-periodic media results in a special
source term that is quasi-periodic but not compactly supported.
The quasi-momentum is uniquely determined by the incident wavenumber and angles.
This scenario presents an opportunity to apply the Limiting Absorption Principle
(LAP) to plane wave scattering in quasi-periodic spaces. Since the quasi-momentum is
fixed, singular perturbation arguments can be adapted to analyze the limiting
solution when the real-valued wavenumber is perturbed by a purely imaginary number.
However, careful treatment of the wavenumber-dependent quasi-periodic space is
required. To address this, we transform the quasi-periodic diffraction problem into
periodic functional spaces before applying the LAP to its equivalent form. Our analysis extends naturally to impenetrable bi-periodic gratings with the
Dirichlet or Neumann boundary conditions. This holds particularly when the quasi-
periodic momentum coincides with a  {propagative wave vector}.
While our earlier work \cite{HK24} applied the LAP to plane wave scattering from a
Dirichlet-type periodic curve, this work appears to be the first treatment of plane
wave scattering by three-dimensional bi-periodic layers.
\smalf
The remainder of this work is organized as follows. Section \ref{sec2} formulates the
three-dimensional grating diffraction problem and establishes the Limiting Absorption
Principle (LAP) for exceptional incident wavenumbers. In Section \ref{sec3}, we
extend the LAP to the time-harmonic electromagnetic scattering problem. The one-
variable singular perturbation arguments together with an auxiliary result are
presented in the Appendix.

\section{Well-posedness for acoustic scattering problem}\label{sec2}

\subsection{Problem formulation}

Assume an acoustic incoming wave is incident onto an inhomogeneous bi-periodic layer
$\R^2\times(-h,h)$ (for some $h>0$) from above. The media above and below the
layer are supposed to be homogeneous and isotropic with the same wavenumber. The
incident wave is supposed to be a time-harmonic plane wave of the form
$u^{in}(x) \exp(-i\omega t)$ with the angular frequency $\omega>0$ and the speed of
sound $c_0 > 0$, where the spatially dependent function  $u^{in}$ takes the form
\be
u^{in}(x)= e^{ik\hat{\theta}\cdot x},\ x\in\R^3,\quad\text{where}\quad
\hat{\theta}=(\sin\theta_1\cos\theta_2,
\sin\theta_1\sin\theta_2,-\cos\theta_1)^\top. \label{plane-wave}
\en
In (\ref{plane-wave}), $\hat{\theta}\in \mathbb{S}^2_-:=\{x\in \R^3: |x|=1, x_3<0\}$
denotes the incident direction from above with the incident angles $\theta_1\in
(-\pi/2,\pi/2)$, $\theta_2\in [0,2\pi)$, and $k:=\omega/c_0>0$ is the
wavenumber of the homogeneous background medium. \tcr{Throughout the paper we set $\tilde{x}=(x_1,x_2)$ and define
 quasi-periodic functions as follows.
\begin{definition}
A function $u(\tilde x, x_3): \R^3\rightarrow \C^3$ is called $\alpha$-quasi-periodic with respect to the parameter $\alpha\in \R^2$ if $e^{-i\alpha\cdot \tilde{x}} u(\tilde{x}, x_3)$ is $2\pi$-periodic with respect to $\tilde{x}$.
\end{definition}
We observe that
the incident field $u^{in}$ can be written as $u^{in}(x)=
e^{ik\tilde\theta\cdot\tilde x}e^{-ik\cos\theta x_3}$ where
$\tilde\theta=\sin\theta_1\binom{\cos\theta_2}{\sin\theta_2}\in\R^2$.
Therefore, the incident field is quasi-periodic with respect to the incidence parameter
$\alpha=k\tilde\theta\in \R^2 $.
Let $q\in L^\infty(\R^3)$ be a positive index of refraction, which is
periodic with respect to $\tilde x$.} Without loss of generality we assume
the periods to be $2\pi$ in both directions. Furthermore, we assume that
$q\equiv 1$ for $|x_3|>h$.
The propagation of the total field $u$ is governed by the Helmholtz equation
\begin{equation} \label{eq:Helmholtz}
\Delta u+k^2 q u=0\quad\mbox{in }\R^3\,.
\end{equation}
The total field $u$ is decomposed into
\begin{align*}
u=\left\{\begin{array}{lll}
u^{in}+u^{sc}, & x_3>h, \\ u^t, &  x_3<-h, \end{array} \right.
\end{align*}
where $u^{sc}$ denotes the wave fields scattered back into $x_3>h$ and $u^{t}$ the
wave fields transmitted into the lower half space $x_3<-h$.
The above model must be complemented with appropriate radiation conditions
for $u^{sc}$ and $u^t$ as $x_3\rightarrow\pm\infty$.
\tcr{Since the incident field is $\alpha$-quasi-periodic}, we expect that the fields $u^{sc}$ and $u^{t}$
are also quasi-periodic with same parameter. Therefore, with
$$ Q_\infty\ :=\ (0,2\pi)^2\times\R $$
we search for a solution in the space $H^1_{\alpha,\loc}(Q_\infty)$, defined as
$$ H^1_{\alpha,loc}(Q_\infty)\ :=\ \bigl\{\phi|_{Q_\infty}:\phi\in H^1_{loc}(\R^3),
\ \tilde x\mapsto e^{-i\alpha\cdot\tilde x}\phi(\tilde x,x_3)
\mbox{ is $2\pi$-periodic}\bigr\}\,, $$
and require the so-called quasi-periodic Rayleigh expansion of the outgoing waves,
i.e.
\begin{equation} \label{eq:Rayleigh}
u(\tilde x,x_3)\ =\ \left\{\begin{array}{cl}
u^{in}(x)+\sum\limits_{n\in\Z^2}u^+_n\,e^{i\alpha_n\cdot \tilde{x}
+i\beta_n (x_3-h)} & \text{ for }x_3>h\,, \\
\sum\limits_{n\in\Z^2}u^{-}_n\,e^{i\alpha_n\cdot \tilde{x}-
i\beta_n (x_3+h)} & \text{ for }x_3<-h\,,\end{array} \right.
\end{equation}
for some $u_n^\pm\in\C$, $n=(n_1,n_2)\in \Z^2$, where $\alpha_n:=
n+\alpha=n+k\tilde\theta $ and , because $|\tilde\theta|=|\sin\theta_1|$,
\begin{equation}\label{eqn:betaN}
\begin{split}
\beta_n\ =\ & \beta_n(k)\ :=\ \sqrt{k^2-|\alpha_n|^2}\ =\
\sqrt{k^2-|n+k\tilde\theta|^2} \\
=\ & \sqrt{k^2\cos^2\theta_1-2k\,n\cdot\tilde\theta-|n|^2}\,.
\end{split}
\end{equation}

\begin{figure}[h]
\centering
\hspace*{-1.5cm}\includegraphics[width=10cm,height=8cm]{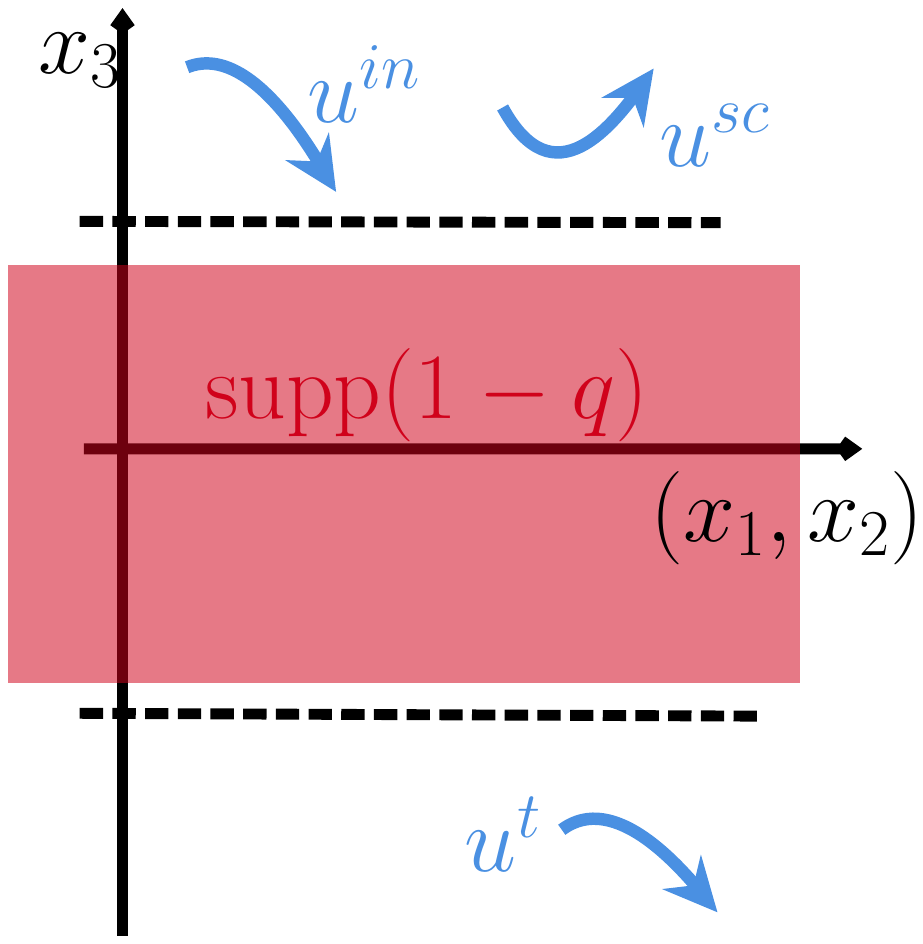}
\vspace*{-1.8cm}
\caption{Diffraction of an acoustic plane wave from a penetrable bi-periodic layer
in $\R^3$.}
\end{figure}

Here and in all of the paper, we choose the square root function to be holomorphic
in the cutted plane $\C\setminus(i\R_{\leq 0})$. In particular,
$\sqrt{t}=i\sqrt{|t|}$ for $t\in\R_{<0}$. Then the convergence of the
series \eqref{eq:Rayleigh} is uniform for $|x_3|\geq h+\eps$ and for every $\eps>0$.
\smalf
We do not indicate the dependence
on $\tilde\theta=\sin\theta_1\binom{\cos\theta_2}{\sin\theta_2}$ because it is kept
fixed. Although the wavenumber $k$ is real we note that the definition of $\beta_n$
allows also complex values of $k$ as the following lemma shows.
\begin{lemma} \label{lem:beta}
If $\alpha=k\tilde{\theta}$, then $\inf\{\Im\beta_n(k+i\eps):n\in\Z^2\}>0$ for all $k,\eps>0$.
\end{lemma}
\begin{proof} First we note that
\begin{equation}
\begin{split}
[\beta_n(k+i\eps)]^2\ =\ & (k+i\eps)^2\cos^2\theta_1-2(k+i\eps)\,n\cdot\tilde\theta
-|n|^2 \\
=\ & [(k^2-\eps^2)\cos^2\theta_1-2k\,n\cdot\tilde\theta-|n|^2]\ +\ 2i\eps\,[
k\cos^2\theta_1-n\cdot\tilde\theta]\,.
\end{split}
\end{equation}
If the real part $\Re[\beta_n(k+i\eps)]^2=
(k^2-\eps^2)\cos^2\theta_1-2k\,n\cdot\tilde\theta-|n|^2<0$ then
$\Im\beta_n(k+i\eps)>0$ by the choice of the square root. Let now
$\Re[\beta_n(k+i\eps)]^2\geq 0$. Then
$k^2\cos^2\theta_1-2k\,n\cdot\tilde\theta-|n|^2>0$ and thus
$n\cdot\tilde\theta<\frac{1}{2}k\cos^2\theta_1$ which proves that
$\Im[\beta_n(k+i\eps)]^2=2\eps\,[k\cos^2\theta_1-n\cdot\tilde\theta]>0$.
Therefore, also in this case $\Im\beta_n(k+i\eps)>0$. Finally, for large values of
$|n|$ we have $\beta_n(k+i\eps)\approx i|n|$, i.e. there exists $n_0\in\N$ with
$\Im\beta_n(k+i\eps)\geq 1$ for $|n|\geq n_0$. Since the number of $n\in\Z^2$
with $|n|<n_0$ is finite the lemma is proven.
\end{proof}
Often we will identify a function $u\in H^1_{\alpha,\loc}(Q_\infty)$ with its
quasi-periodic extension to $\R^3$.

It will be necessary to transform the problem for the quasi-periodic solution $u$
to a problem for the periodic (with respect to $\tilde x$) solution
$v(x):=e^{-i\alpha\cdot\tilde x}u(x)=e^{-ik\tilde\theta\cdot\tilde x}u(x)$ which
belongs to the space $H^1_{\per,loc}(Q_\infty)$, defined as $H^1_{\alpha,loc}(Q_\infty)$
for $\alpha=(0,0)$, i.e.
$$ H^1_{\per,loc}(Q_\infty)\ :=\ \bigl\{\phi|_{Q_\infty}:\phi\in H^1_{loc}(\R^3),\
\tilde x\mapsto\phi(\tilde x,x_3)\mbox{ is $2\pi$-periodic}\bigr\}\,. $$
If $u\in H^1_{\alpha,loc}(Q_\infty)$ satisfies \eqref{eq:Helmholtz}, \eqref{eq:Rayleigh}
then $v\in H^1_{\per,loc}(Q_\infty)$ satisfies
\begin{equation} \label{eq:Hper}
\Delta v+2ik\tilde\theta\cdot\nabla_{\tilde{x}} u+k^2(q-\sin^2\theta_1)v=0
\quad\mbox{in}\quad\R^3\,,\quad \tcr{\nabla_{\tilde x}u:=(\partial_1 u, \partial_2 u)\,,}
\end{equation}
complemented by the periodic upward and downward Rayleigh expansions:
\begin{equation}\label{eq:Rayleigh-per}
v(\tilde{x}, x_3)\ =\ \left\{\begin{array}{cl}
e^{-ik\cos\theta_1 x_3}+\sum\limits_{n\in\Z^2}v_n^+\,e^{in\cdot \tilde{x}+
i\beta_n (x_3-h)} &\mbox{ for }x_3>h\,, \\
\sum\limits_{n\in\Z^2}v_n^-\,e^{in\cdot \tilde{x}-
i\beta_n (x_3+h)} & \mbox{ for }x_3<-h\,,
\end{array}\right.
\end{equation}
where $\beta_n=\beta_n(k)$ is again given by \eqref{eqn:betaN}. Since $\beta_n$ has
an extension to $k\in\C$ with $\Im k>0$ we note already here that this periodic
version allows complex values of $k$ (with $\Im k\geq 0$).

For $h>0$ and an arbitrary $\alpha\in\R^2$ set
\begin{eqnarray*}
Q_h & = & (0,2\pi)^2\times(-h, h)\,, \\
H^1_{\alpha}(Q_h) & := & \bigl\{\phi\in H^1(Q_h):\; \tilde x\mapsto
e^{-i\alpha\cdot\tilde x}\phi(\tilde x,x_3)\mbox{ is $2\pi$-periodic}\bigr\},\\
H^1_\per(Q_h) & := & \bigl\{\phi\in H^1(Q_h):\; \tilde x\mapsto
\phi(\tilde x,x_3)\mbox{ is $2\pi$-periodic}\bigr\}\,.
\end{eqnarray*}
Given an incident direction $\hat\theta$, it is well known that the scattering problem
\eqref{eq:Helmholtz}, \eqref{eq:Rayleigh} admits a $\alpha-$quasi-periodic solution
(with $\alpha=k\tilde\theta$) for all $k\in\R_{>0}$. Furthermore, there exists a
discrete set $\mathcal{D}\subset\R_{>0}$ with the only accumulating point at infinity
such that the solution is unique for $k\in \R_{>0}\setminus\mathcal{D}$, see, e.g.,
\cite{BBS94,K93}. The discrete set $\mathcal{D}$ consists of those $k>0$ for which
$\alpha=k\tilde\theta =k\sin\theta_1\binom{\cos\theta_2}{\sin\theta_2}$ is a propagative wave
vector \tcr{defined in Definition \ref{def:critical} (ii) below. }

\begin{definition} \label{def:critical}
\tcr{Let $\alpha\in \R^2$ be a general vector.}
	(i) $\alpha\in \R^2$ is called a \tcr{cut-off value or} cut-off vector if there exists
$n=(n_1,n_2)\in\Z^2$ such that $|\alpha+n|=k$.
\smalf
(ii) $\alpha\in \R^2$ is called a  {propagative wave vector} if there exists a
non-trivial function $\phi\in H^1_{\alpha,loc}(Q_\infty)$ such that
\begin{equation} \label{exc:a}
\Delta\phi + k^2 q \phi\ =\ 0\text{ in }\; \R^3\,,
\end{equation}
and $\phi$ satisfies the upward and downward Rayleigh expansions \eqref{eq:Rayleigh}
for $\pm x_3>h$ where $\beta_n=\sqrt{k^2-|n+\alpha|^2}$. For a  {propagative wave vector} of
the form $\alpha=k\tilde\theta$ we denote the corresponding space of modes by
\begin{equation} \label{eq:mode}
\cM_k\ =\ \bigl\{\phi\in H^1_\alpha(Q_\infty):
\Delta\phi+k^2q\phi=0\text{ in }\R^3\,,\ \phi\text{ satisfies }
\eqref{eq:Rayleigh} \bigr\}\,.
\end{equation}
\end{definition}

From the definition it is obvious that $\alpha\in\R^2$ is a cut-off value or
 {propagative wave vector} if, and only, if $\alpha+\ell$ is a cut-off value or
 {propagative wave vector}, respectively, for all $\ell\in\Z^2$. We illustrate this with a
simple example of a homogeneous layer.
\begin{example} \label{ex:1}
\tcr{Let $k>0$ and $\alpha\in\R^2$ with $|\alpha|>k$.  Consider the height $h=1$ and the constant}
refractive index $q>1$ in the layer $\R^2\times(-1,1)$. We make an  ansatz for a mode
in the form
$$ \phi(x)\ =\ e^{i\alpha\cdot\tilde x}\cdot \left\{\begin{array}{cl}
\cos\sqrt{k^2q-|\alpha|^2}\,e^{-\sqrt{|\alpha|^2-k^2}(x_3-1)}, & x_3>1\,,  \\
\cos(\sqrt{k^2q-|\alpha|^2}x_3)\,, & |x_3|<1\,, \\
\cos\sqrt{k^2q-|\alpha|^2}\,e^{\sqrt{|\alpha|^2-k^2}(x_3+1)}, & x_3<-1\,.
\end{array}\right. $$
We have to choose the parameters $k$, $\alpha$, and $q$ such that the derivative
with respect to $x_3$ is continuous for $x_3=\pm 1$. This leads to the equation
$$ \sqrt{|\alpha|^2-k^2}\,\cos\sqrt{k^2n-|\alpha|^2}\ =\
\sqrt{k^2n-|\alpha|^2}\,\sin\sqrt{k^2n-|\alpha|^2}\,. $$
For a particular example we choose the parameters such that
$\sqrt{k^2n-|\alpha|^2}=\pi/4$ (then $\sin\sqrt{k^2n-|\alpha|^2}=
\cos\sqrt{k^2n-|\alpha|^2}$) and  $\sqrt{|\alpha|^2-k^2}=\sqrt{k^2n-|\alpha|^2}$.
For $q=2$ this leads to $k=\pi/(2\sqrt{2})$ and $|\alpha|=\pi\sqrt{3}/4$. The set
$\{\alpha\in[-1/2,1/2]^2:\exists\ell\in\Z^2:|\alpha+\ell|=k\}$ of cut-off vectors
the set $\{\alpha\in[-1/2,1/2]^2:\exists\ell\in\Z^2:|\alpha+\ell|=\pi\sqrt{3}/4\}$ of
 {propagative wave vector}s (both restricted to $[-1/2,1/2]\times[-1/2,1/2]$) are shown in
Figure~2.
\begin{figure} \label{fig:2}
\includegraphics[width=0.4\textwidth]{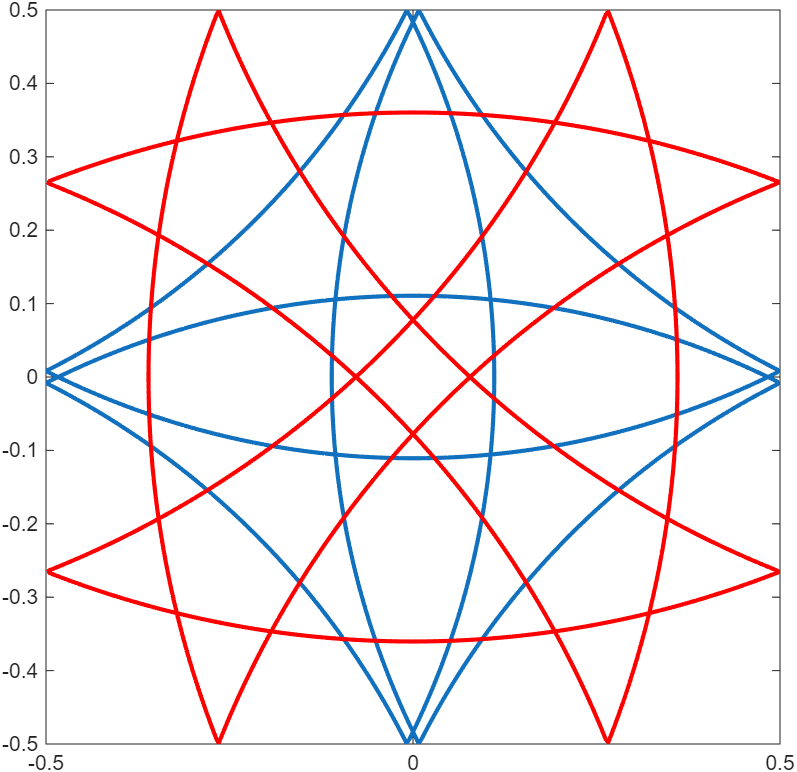}
\caption{The sets of cut-off vectors (blue) and  {propagative wave vector}s (red) for
Example~\ref{ex:1}}
\end{figure}
\end{example}
Therefore, if $\alpha$ is a  {propagative wave vector}, the homogeneous scattering problem with
$u^{in}=0$ admits non-trivial solutions in $H^1_{\alpha,\loc}(Q_\infty)$ satisfying the
Rayleigh condition. The aim of this paper is to investigate uniqueness and existence of
the diffraction problem when $\alpha$ of the form $\alpha=k\tilde\theta$ is a propagative wave
vector but not a cut-off vector. The functions $\phi=\phi(\cdot,\alpha,k)$ to the
homogeneous problem are usually referred to as guided wave modes, which decay
exponentially in the $x_3$-direction; see Lemma \ref{lem:L-prop} (i) below. In physics,
the guided mode $\phi$ of \eqref{exc:a} is called a Bound State in the Continuum (BIC)
if $|\alpha|<k$. In this sense, $\alpha$ being a  {propagative wave vector} implies that the bi-periodic index function $q$ is a BIC-supporting material at the energy $k>0$. Propagative wave
vectors do not exist if the refractive index function $q$ fulfills certain monotonicity
properties, which exclude guided/surface waves propagating along the directions of
periodicity (\cite{BBS94,CZ98, ES98, K95}). In general cases, the relation between the
 {propagative wave vector}s $\alpha\in \R^2$ and the wavenumber $k\in \R$ constitutes the
dispersion map for bi-periodic media.

In the subsequent subsections we fix the incident direction $\hat{\theta}\in
\mathbb{S}_-^2$. Throughout this paper we assume that $k\tilde\theta$ is not a
cut-off vector if it is a  {propagative wave vector}.

\begin{assumption} \label{assump1}
 $|n+\alpha|\not=k$ for every $n\in\Z^2$ if $\alpha=k\tilde\theta=
 k\sin\theta_1(\cos\theta_2\sin\theta_2)^\top$ is a  {propagative wave vector}.
\end{assumption}

\subsection{Variational formulation} \tcr{Let us consider the $\alpha$-quasiperiodic incident plane wave \eqref{plane-wave}, where $\alpha=k\tilde\theta$ and $\tilde\theta\in \R^2$ are kept fixed.} Set
$\Gamma_h:=\{ (x_1, x_2,h): 0<x_1, x_2<2\pi\}$. For $k>0$, i.e.,
$\alpha=k\tilde\theta\in\R^2$, we introduce the $\alpha$-quasi-periodic and periodic,
respectively, Sobolev spaces on $\Gamma_h$  by
\ben
H_\alpha^{1/2}(\Gamma_h) & := & \{f\in H^{1/2}(\Gamma_h):e^{-i\alpha\cdot
\tilde{x}}f(\tilde{x})\mbox{ is $2\pi$-periodic in $\tilde{x}$}\}\,, \\
H_\per^{1/2}(\Gamma_h) & := & \{f\in H^{1/2}(\Gamma_h):f(\tilde{x})\mbox{ is $2\pi$-periodic
in $\tilde{x}$}\}\,.
\enn
Define the periodic and quasi-periodic, respectively, Dirichlet-to-Neumann (DtN) maps
on the artificial boundary $\Gamma_h$ by
\begin{eqnarray}
\hspace*{1cm}(T_k^\pm f)(\tilde{x}) & := & \pm \sum_{n\in\Z^2}i\beta_n\,f_n\,
e^{i n\cdot\tilde{x}},\quad f(\tilde{x})=
\sum_{n\in\Z^2}f_n\,e^{i n\cdot\tilde{x}}\in H_\per^{1/2}(\Gamma_{\pm h})\,,
\label{dtn} \\
\hspace*{1cm}(\tilde{T}^\pm_{k}\tilde{f})(\tilde{x}) & := & \pm \sum_{n\in\Z^2}
i\beta_n\,\tilde{f}_n\,e^{i(n+\alpha)\cdot\tilde{x}},\quad
\tilde{f}(\tilde{x})=\sum_{n\in\Z^2}\tilde{f}_n
e^{i(n+\alpha)\cdot\tilde{x}}\in H_\alpha^{1/2}(\Gamma_{\pm h})\,.
\label{dtn2}
\end{eqnarray}
We indicated the dependence of the operators on $k$ because
$\alpha=k\tilde\theta$ and $\tilde\theta$ is kept fixed. Simple
calculations show that
$$ T_k^\pm: H_\per^{1/2}(\Gamma_{\pm h})\rightarrow
H_\per^{-1/2}(\Gamma_{\pm h}), \quad
\tilde{T}_k^\pm: H_{\alpha}^{1/2}(\Gamma_{\pm h})\rightarrow
H_{\alpha}^{-1/2}(\Gamma_{\pm h}) $$
are linear bounded operators. They are referred to as DtN maps, because
\ben
\tilde{T}_k^+(u^{sc}|_{\Gamma_h})=\frac{\partial u^{sc}(x)}
{\partial x_3}|_{\Gamma_h},\quad \tilde{T}_k^-(u^{t}|_{\Gamma_{-h}})=
\frac{\partial u^{t}(x)}{\partial x_3}|_{\Gamma_{-h}}.
\enn
The following variational formulation for the total field $u\in H^1_{\alpha}(Q_h)$
can be easily derived.
\begin{lemma} \label{lem:equiv}
(a) Let $u\in H^1_{\alpha,\loc}(Q_\infty)$ be a solution of
\eqref{eq:Helmholtz}, \eqref{eq:Rayleigh}. Then $u|_{Q_h}\in H^1_{\alpha}(Q_h)$
solves the variational equation
\begin{equation} \label{eq:var-u}
\tilde{a}_k(u, \phi)\ =\ -\int\limits_{\Gamma_h}
(\tilde{T}_k^+(u^i|_{\Gamma_h})-\partial_\nu u^i)\,\overline{\phi}\,ds\ =\
-2ik\cos\theta_1\, e^{-ikh\cos\theta_1}\int\limits_{\Gamma_h}
e^{ik\tilde\theta\cdot \tilde{x}}\,\overline{\phi}\,ds
\end{equation}
for all $\phi\in H^1_{\alpha}(Q_h)$, where
\begin{equation} \label{atil}
 \tilde{a}_k(u,\phi)\ :=\ \int\limits_{Q_h}\bigl[\nabla u\cdot\overline{\nabla\phi}-
k^2 q u\,\overline{\phi}\bigr]\,dx\ -\ \int\limits_{\Gamma_h}\tilde{T}^+_k u\,
\overline{\phi}\,ds\ +\ \int\limits_{\Gamma_{-h}}\tilde{T}^-_k u\,\overline{\phi}\,ds\,
\end{equation}
for all $u,\phi\in H^1_{\alpha}(Q_h)$. The line integrals in \eqref{atil} are actually
the dual forms $\langle\cdot,\cdot\rangle$ in the dual system
$\langle H^{-1/2}(\Gamma_h),H^{1/2}(\Gamma_h)\rangle$.
\smalf
(b) Let $u\in H^1_{\alpha}(Q_h)$ solve \eqref{eq:var-u}. Extend $u$ to $Q_\infty$
by the Rayleigh expansion \eqref{eq:Rayleigh} where $u_n^+$ are the Fourier
coefficients of $u^{sc}=u-u^{in}$ on $\Gamma_h$ and $u_n^-$ are the Fourier
coefficients of $u^t$ on $\Gamma_{-h}$. Then $u$ solves the scattering
problem \eqref{eq:Helmholtz} \tcr{and} \eqref{eq:Rayleigh}.
\smalf
(c) The same equivalence holds for the periodic problem \eqref{eq:Hper},
\eqref{eq:Rayleigh-per}, where the periodic variational formulation for
$v\in H^1_\per(Q_h)$ reads as follows:
\begin{equation} \label{eq:var_per}
a_k(v,\psi)\ =\ -2ik\cos\theta_1\,e^{-ikh\cos\theta_1}\int\limits_0^{2\pi}
\int\limits_0^{2\pi}\overline{\psi(\tilde{x},h)}\,d\tilde{x}\quad\mbox{for all }
\quad\psi\in H^1_\per(Q_h)
\end{equation}
where now
\ben
a_k(v,\psi) & := & \int\limits_{Q_h} \biggl[\nabla v\cdot\overline{\nabla\psi}-
2i(\alpha\cdot\nabla_{\tilde{x}} v)\overline{\psi}-(k^2q-|\alpha|^2)\,v\,
\overline{\psi}\biggr]\,dx\ \\
& & -\ \int\limits_{\Gamma_h} T_k^+ v\,\overline{\psi}\,ds\ +\
\int\limits_{\Gamma_{-h}}T_k^- v\,\overline{\psi}\,ds \\
& = & \int\limits_{Q_h} \biggl[\nabla v\cdot\overline{\nabla\psi}-
2ik(\tilde\theta\cdot\nabla_{\tilde{x}} v)\overline{\psi}-k^2(q-\sin^2\theta_1)\,
v\,\overline{\psi}\biggr]\,dx \\
& & -\ \int\limits_{\Gamma_h} T_k^+ v\,\overline{\psi}\,ds\ +\
\int\limits_{\Gamma_{-h}}T_k^- v\,\overline{\psi}\,ds
\enn
for all $v,\psi\in H^1_\per(Q_h)$.
\end{lemma}
\begin{remark} \label{rem:cmplx}
It is worth mentioning that the quasi-periodic variational formulation
\eqref{eq:var-u} is equivalent to the original problem \eqref{eq:Helmholtz},
\eqref{eq:Rayleigh} only for real-valued wavenumbers. This is due to the reason that,
when applying integration by parts, the integrals over the vertical boundaries of $Q_h$
cancel out only for real-valued wavenumbers $k>0$. However, the
periodic version \eqref{eq:var_per} is equivalent to \eqref{eq:Hper},
\eqref{eq:Rayleigh-per}, and, as already mentioned above, can be formulated even for
complex wavenumbers $k\in \C$ (with $\Im k\geq 0$) when we use the definition
\eqref{eqn:betaN} of $\beta_n$.
\end{remark}

We equip $H^1_\per(Q_h)$ with the inner product
\begin{equation} \label{InPr}
\langle v,\psi\rangle_{H^1_\per(Q_h)}\ :=\ \int\limits_{Q_h}
\nabla v\cdot\nabla\overline{\psi}\,dx\ +\ 2\pi\sum_{n\in\Z^2}
(1+|n|^2)^{1/2}\bigl(v^+_n\,\overline{\psi^+_n}+v^-_n\,\overline{\psi^-_n} \bigr)\,,
\end{equation}
where $\psi^\pm_n$ and $v^\pm_n$ denote the Fourier coefficients of
$\psi(\tilde{x},\pm h)$ and $v(\tilde{x}, \pm h)$, respectively. By the
representation theorem of Riesz, there exist $f_k\in H^1_\per(Q_h)$ and a
linear bounded operator $L_k$ from $H^1_\per(Q_h)$ into itself with
\begin{eqnarray}
\langle f_k,\psi\rangle_{H^1_\per(Q_h)} & = & -2ik\cos\theta_1\,e^{-ikh\cos\theta_1}
\int\limits_0^{2\pi}\int\limits_0^{2\pi}\overline{\psi(\tilde{x},h)}\,
d\tilde{x}\,, \label{eq:fk} \\
\langle L_kv,\psi\rangle_{H^1_\per(Q_h)} & = & a_k(v,\psi) \nonumber \\
& = & \int\limits_{Q_h}\bigl[\nabla v\cdot\nabla\overline{\psi}-2i\,
k\,[\tilde\theta\cdot\nabla_{\tilde{x}}v]\,\overline{\psi}-
k^2(q-\sin^2\theta_1)\,v\,\overline{\psi}\bigr]\,dx \label{eq:Lk} \\
&& -\ 4\pi^2\sum_{n\in \N} i\beta_n \left(v^+_n\,\overline{\psi^+_n}+v^-_n\,
\overline{\psi^-_n}  \right)  \nonumber
\end{eqnarray}
for all $\psi,v\in H^1_\per(Q_h)$. Then the variational equation \eqref{eq:var_per}
can be rewritten as
\be\label{eq7}
L_kv\ =\ f_k\quad\mbox{in}\quad H^1_\per(Q_h)\,.
\en
Using the compact embedding of $H^1_\per(Q_h)\rightarrow L^2(Q_h)$, one can show
that the operator $K_k:=I-L_k$, given by
\ben
\langle K_k\,v,\psi\rangle_{H^1_\per(Q_h)}
& := & \int\limits_{Q_h}\bigl[2i\,k\,[\tilde\theta\cdot
\nabla_{\tilde{x}}v]\,\overline{\psi}+k^2(q-\sin^2\theta_1)\,v\,\overline{\psi}
\bigr]\,dx \\
&& +\ 4\pi^2\sum_{n\in \N^2}(\sqrt{1+n^2}+ i\beta_n)
\left(v^+_n\,\overline{\psi^+_n}+v^-_n\,\overline{\psi^-_n}  \right),
\enn
for all $\psi,v\in H^1_\per(Q_h)$, is compact as an operator from $H^1_\per(Q_h)$
into itself. Therefore, $L_k$ is a Fredholm operator with index zero.
\medlf
Further properties of the operator $L_k$ are collected in the following lemma.
\begin{lemma}\label{lem:L-prop}
\begin{itemize}
\item[(i)] For $k>0$, the operator equation \eqref{eq7} admits at least one solution
$v\in H^1_\per(Q_h)$.
The operator $L_k$ is an isomorphism from  $H^1_\per(Q_h)$ onto itself if,
and also if, $\alpha$ is not a  {propagative wave vector}. If $\alpha$ is a
 {propagative wave vector} (and Assumption~\ref{assump1} holds) then the
null space $\mathcal{N}:=\mathcal{N}(L_k)=\mathcal{N}(L_k^\ast)$
is finite dimensional and consists of surface wave modes only, i.e., any
$v\in\cN$ has an extension to $Q_\infty$ with
\begin{equation} \label{eq:evanescent}
v(x)\ =\ \sum_{n\in\Z^2:|n+\alpha|>k}v^\pm _n\,
e^{in\cdot \tilde{x}\mp|\beta_n|(x_3\mp h)}\,,\quad \pm x_3>h\,.
\end{equation}
Therefore, $v\in\cN$ if, and only if, the quasi-periodic extension of
$v(x)e^{ik\tilde\theta\cdot\tilde x}$ belongs to the mode space $\cM_k$.
\item[(ii)] The Riesz number of $L_k$ is one, that is, $\mathcal{N}(L_k)=
\mathcal{N}(L_k^2)$. Moreover, the decomposition
$H^1_\per(Q_h)=\mathcal{N}(L_k)\oplus\mathcal{R}(L_k)$ is orthogonal. Here
$\mathcal{R}(L_k)$ denotes the range of the operator $L_k$.
\item[(iii)] Suppose $\Im k>0$ and
$q(x)\geq \sin^2\theta_1$ in $Q_\infty$. Then $L_k$ (which is well defined
by Remark~\ref{rem:cmplx}) is an isomorphism, and there is a unique solution
$v\in H^1_{\per,loc}(Q_\infty)$ to \eqref{eq:Hper}, \eqref{eq:Rayleigh-per}.
\end{itemize}
\end{lemma}
\begin{proof}
(i) If $L_kv=0$ then $ \tilde a_k(u,u)=0$ where $u=ve^{i\alpha\cdot\tilde x}$. Taking
the imaginary part and using the definition of $\tilde{T}_{k}^\pm$ yields
$v_n^\pm=0$ for all $n$ with $|n+\alpha|<k$.
\newline
The adjoint operator $L_{k}^{*}$ of $L_{k}$ is defined by
$$ \langle L_{k}^{*}v,\psi\rangle_{H^1_\per(Q_h)}\ =\
\langle v,L_{k}\psi\rangle_{H^1_\per(Q_h)}\ =\
\overline{\langle L_{k}\psi,v\rangle}_{H^1_\per(Q_h)}\ =\ \overline{a_{k}(\psi,v)} $$
for all $v,\psi\in H_\per^1(Q_h)$. Furthermore, using \eqref{eq:evanescent} we obtain
$$ \langle L_{k}v,\psi\rangle_{H^1_\per(Q_h)}\ =\ a_k(v,\psi)=\overline{a_{k}(\psi,v)}\
=\ \langle L_{k}^\ast v,\psi\rangle_{H^1_\per(Q_h)}\qquad \mbox{for}\quad
v\in\cN(L_k) $$
by simple calculations. From this we conclude that $\cN(L_k)=\cN(L_k^\ast).$ The
existence of a solution $v\in H_\per^1(Q_h)$ follows from the fact that
$\langle f_k,\psi\rangle_{H^1_\per(Q_h)}=0$ for all $\psi\in\mathcal{N}$ and
the Fredholm alternative. Indeed, from the representation \eqref{eq:evanescent} for
$\psi$ and $|\alpha|<k$ we conclude that $\int_0^{2\pi}\int_0^{2\pi}
\psi(\tilde x,h)\,d\tilde x=4\pi^2\psi_{(0,0)}=0$. The space $\mathcal{N}$ is finite
dimensional because $K_k$ is compact.
\smalf
(ii) It is obvious that $\mathcal{N}(L_k)\subset\mathcal{N}(L_k^2).$ To prove the
reverse direction, we assume $L_k^2w=0$ for some $w\in H_\per^1(Q_h)$ and set
$v=L_{k}w\in\mathcal{R}(L_{k}).$ Since $v\in\mathcal{N}(L_k)=\mathcal{N}(L_k^*)$, we
obtain
$$\|v\|^2_{H^1_\per(Q_h)}\ =\ \langle v,v\rangle_{H^1_\per(Q_h)}\ =\
\langle v,L_kw\rangle_{H^1_\per(Q_h)}\ =\ \langle L_k^*v,w\rangle_{H^1_\per(Q_h)}\
=\ 0\,, $$
which proves $\cN(L_k^2)\subset\cN(L_k)$ and thus the coincidence
$\cN(L_k^2)=\cN(L_k)$. This also implies $\cN\cap\cR(L_k)=\emptyset$ and hence
$H_{per,0}^1(Q_h)=\cN\oplus\cR(L_k)$. The orthogonality follows from the relation
$\cN(L_k)=\cN(L_k^\ast)$.
\smalf
(iii) First we note that $L_k$ is a Fredholm operator also for $\Im k>0$ by the same
arguments as for the case of real $k$. Therefore, it suffices to show injectivity.
Let $\Im k>0$ and $v\in H^1_\per(Q_h)$ be a periodic solution of the homogeneous problem
$L_kv=0$. We need to show $v\equiv 0$. We extend $v$ into $Q_\infty\setminus Q_h$ by
the Rayleigh expansion
$$ v(x)\ =\ \sum_{n\in \Z^2}v_n^\pm\,
e^{\pm i\beta_n(x_3\mp h)+i n\cdot \tilde{x}}\,,\quad \pm x_3>h\,. $$
By Lemma~\ref{lem:beta} the imaginary part of $\beta_n$ has a positive lower
bound, that is, $\Im\beta_n\geq\delta>0$ for all $n\in \Z^2$. Therefore,
$v$ is evanescent, i.e. $v\in H^1_\per(Q_\infty)$ and satisfies \eqref{eq:Hper}
in $\R^3$. For any $\psi\in H^1_\per(Q_\infty)$ we apply Green's theorem in $Q_\infty$
and obtain
$$ \int\limits_{Q_\infty}\bigl[\nabla v\cdot\nabla\overline{\psi}-2i\,k\,
[\tilde\theta\cdot\nabla_{\tilde{x}}v]\,\overline{\psi}-
k^2(q-\sin^2\theta_1)\,v\,\overline{\psi}\bigr]\,dx\ =\ 0 $$
for all $\psi\in H^1_\per(Q_\infty)$. This can be written as the quadratic form
$$ Av -kBv-k^2Cv=0 $$
where $A$, $B$, and $C$ do not depend on $k$ and are self-adjoint bounded operators
from $H^1_\per(Q_\infty)$ into itself defined by
\begin{eqnarray*}
\langle Av,\psi\rangle_{H^1_\per(Q_\infty)} & := & \int\limits_{Q_\infty}
\nabla v\cdot\nabla\overline{\psi}\,dx,\qquad
\langle Bv, \psi\rangle_{H^1_\per(Q_\infty)}\ :=\ 2i\int\limits_{Q_\infty}[\tilde\theta
\cdot\nabla_{\tilde{x}}v]\,\overline{\psi}\,dx\,, \\
\langle Cv, \psi\rangle_{H^1_\per(Q_\infty)} & := & \int\limits_{Q_\infty}
(q-\sin^2\theta_1)v\,\overline{\psi}\,dx\quad\text{for }v,\psi\in H^1_\per(Q_\infty)\,.
\end{eqnarray*}
Here, $\langle\cdot,\cdot\rangle_{H^1_\per(Q_\infty)}$ denotes the usual inner product
in $H^1(Q_\infty)$. Note that $A$ is positive on $H^1_\per(Q_\infty)$, that is,
$\langle Av,v\rangle_{H^1_\per(Q_\infty)}>0$
for all $v\not=0$ (because the constant functions do not belong to
$H^1_\per(Q_\infty)$).
%\cblue{Alternative way to prove uniqueness without using the square root (replace green
%part): Multiplying the equation $Av -kBv-k^2Cv=0$ by $v$ and taking the real- and
%imaginary part yields (with $a=\langle Av,v\rangle$, $b=\langle Bv,v\rangle$, and
%$c=\langle Cv,v\rangle$)
%$$ a-b\Re k-c\bigl[(\Re k)^2-(\Im k)^2\bigr]=0\,,\quad b(\Im k)+2c(\Re k)(\Im k)=0\,. $$
%Division of the second equation by $\Im k>0$ yields $b+2c(\Re k)=0$.
%We eliminate $b$ and arrive at $a+c(\Re k)^2+c(\Im k)^2=0$. If $c\geq 0$ -- which is the
%case because $q\geq \sin^2\theta_1$ -- this implies $a=0$, i.e. $v=0$. On the other hand, if
%$c<0$ then, by Friedrich's inequality, $|c|\leq(\Vert q\Vert_\infty+1)\rho\,a$ for some
%$\rho>0$ which is independent of $v$. Then $0=a+c|k|^2\geq a-
%(\Vert q\Vert_\infty+1)\rho|k|^2a=\bigl[1-(\Vert q\Vert_\infty+1)\rho|k|^2\bigr]a$ which
%yields $a=0$ if $|k|$ is sufficiently small.
%}
Therefore, $C$ has
a square root $W$; that is, a bounded and self-adjoint operator $W$ satisfying
$W^2=C$. Then we have $Av-kBv-k^2W^2v=0$. Setting $V=(v_1, v_2)^\top$ with $v_1=v$
and $v_2=kWv$ we obtain the equivalent (since $k\not=0$) system
$\mathcal{A}V=k\mathcal{B}V$, where
$$ \mathcal{A}:= \left(\begin{array}{cc} A & 0 \\ 0 & I \end{array}\right),\quad
\mathcal{B}:=\left(\begin{array}{cc} B & W \\ W & 0 \end{array}\right), $$
and thus $\langle\mathcal{A}V, V\rangle=k\,\langle\mathcal{B}V, V\rangle$.
Since the operator matrices $\mathcal{A}$ and $\mathcal{B}$ are both self-adjoint in
$H^1_\per(Q_\infty)\times H^1_\per(Q_\infty)$ we obtain
$(\Im k)\,\langle\mathcal{B}V, V\rangle=0$, thus $\langle\mathcal{B}V,V\rangle=0$
because $\Im k>0$, thus $\langle\mathcal{A}V,V\rangle=0$, and thus $V=0$. This proves
uniqueness for a complex-valued wavenumber with a positive imaginary part.
\end{proof}

\begin{remark} There exist other approaches to prove uniqueness when the background
medium is absorbing (i.e., $\Im k>0$). In two dimensions, the uniqueness was
justified in \cite{CR} for Dirichlet rough surface scattering problems with a weak
regularity assumption on the boundary.
\end{remark}

\subsection{The limiting absorption argument}

The well-posedness of the scattering problem at a  {propagative wave vector} which is
not a cut-off vector is essentially based on the Limiting Absorption Principle
(LAP) by investigating the limit $u(x; k+i\epsilon)$ as $\epsilon\rightarrow 0+$.
Here $u(x;k+i\epsilon)$ denotes the unique solution to the diffraction problem when
$k>0$ is replaced by $k+i\epsilon$.
For this purpose we need to apply an abstract singular perturbation result from
\cite[Theorem 2.7 and Remark 2.8]{K22p}, which is stated in Lemma \ref{lem:sing-pert} in the
Appendix.
\smalf
To apply Lemma \ref{lem:sing-pert} to the operator equation \eqref{eq7}, we assume that
$\alpha=k\tilde\theta$ is a  {propagative wave vector} (but not a cut-off vector). We set
$X=H^1_\per(Q_h)$ and $\cN=\cN(L_k)$ and denote by $P: H^1_\per(Q_h)
\rightarrow\mathcal{N}$ the projection operator. Set
$f(\epsilon):=f_{(k+i\epsilon)}$ and $L(\epsilon)=L_{(k+i\epsilon)}$. For
$\psi\in H^1_\per(Q_h)$, it follows from the definitions of $f_k$ and $L_k$
(see \eqref{eq:fk} and \eqref{eq:Lk}) that
$$ \langle f(\epsilon),\psi\rangle_{H^1_\per(Q_h)}\ =\ -2i(k+i\epsilon)\cos\theta_1\,
e^{-i(k+i\epsilon)h\cos\theta_1}\int\limits_0^{2\pi}\int\limits_0^{2\pi}
\overline{\psi(\tilde{x},h)}\,dx_1\,dx_2, $$
\begin{eqnarray*}
& & \langle L(\epsilon)v,\psi\rangle_{H^1_\per(Q_h)} \\
& = & \int\limits_{Q_h}\bigl[
\nabla v\cdot\nabla\overline{\psi}-2i(k+i\epsilon)\,[\tilde\theta\cdot
\nabla_{\tilde{x}}v]\,\overline{\psi}\bigr]\,dx\ -
\int\limits_{Q_h}\bigl[(k+i\epsilon)^2(q-\sin^2\theta_1)\,v\,\overline{\psi}\bigr]\,
dx \\
& & -\ i\,4\pi^2 \sum_{n\in \Z^2} \sqrt{(k+i\epsilon)^2\cos^2\theta_1-
2(k+i\epsilon)\,n\cdot\tilde\theta-|n|^2}\,\left(v^+_n\,\overline{\psi^+_n}+
v^-_n\,\overline{\psi^-_n}\right)
\end{eqnarray*}
for $v,\psi\in X$ with Fourier coefficients $v_n^\pm=\frac{1}{2\pi}
\int_0^{2\pi}\int_0^{2\pi}v(\tilde x,\pm h)\,e^{-in\cdot\tilde x}d\tilde x$ and,
analogously, $\psi_n^\pm$. In the series we used the form \eqref{eqn:betaN} of
$\beta_n$ for $k$ replaced by $k+i\epsilon$.

From the above expressions we observe that $f(\epsilon)$ and $L(\epsilon)$ are
differentiable with respect to $\epsilon$ in a neighborhood of $0$, provided
$\alpha=k\tilde\theta=k\sin\theta_1 (\cos\theta_2,\sin\theta_2)$ is not a
cut-off vector. Below we show properties of $f(\epsilon)$ and $L(\epsilon)$.

\begin{lemma}\label{lem:DL-prop}
\begin{itemize}
\item[(i)] $f(\epsilon)\in\mathcal{R}(L(\epsilon))$ for all $\epsilon>0$ and
$f(0),f^\prime(0)\in\mathcal{R}=\mathcal{R}(L(0))$.
\item[(ii)] $PL^\prime(0)$ is one-to-one on $\cN$, if $q(x)\geq \sin^2\theta_1$ in
$\R^3$.
\end{itemize}
\end{lemma}
\textbf{Proof:} (i) By Lemma \ref{lem:L-prop} (iii),  $L(\epsilon)$ is invertible  and
thus $f(\epsilon)\in\mathcal{R}(L(\epsilon))$ for all $\epsilon>0$. On the other
hand,
we have $f(0)\in \mathcal{R}$, because by Lemma \ref{lem:L-prop} (i) and (ii), $f(0)$ is
orthogonal to $\mathcal{N}$ and $X$ admits the orthogonal decomposition
$X=\mathcal{N}\oplus\mathcal{R}$.
Since the null space $\mathcal{N}$ consists of evanescent wave modes only we conclude
\begin{eqnarray*}
\langle Pf(0),\psi\rangle_{H^1_\per(Q_h)} & = & \langle f(0),\psi\rangle_{H^1_\per(Q_h)}
= -2ik\cos\theta_1\,e^{-ikh \cos\theta_1}\int\limits_0^{2\pi}\int\limits_0^{2\pi}
\overline{\psi(\tilde{x}, h)}\, d\tilde x\ =\ 0\,, \\
\langle Pf^\prime(0),\psi\rangle_{H^1_\per(Q_h)} & = &
\langle f^\prime(0),\psi\rangle_{H^1_\per(Q_h)} \\
& = & 2\cos\theta_1(1-ikh\cos\theta_1)\,e^{-ikh\cos\theta_1}
\int\limits_0^{2\pi}\int\limits_0^{2\pi}\overline{\psi(\tilde{x}, h)}\,d\tilde x\
=\ 0
\end{eqnarray*}
for all $\psi\in\mathcal{N}$. Note that $\int_0^{2\pi}\int_0^{2\pi}
\psi(\tilde x,h)\,d\tilde x=4\pi^2\psi_{(0,0)}$ vanishes because $\psi$ is
evanescent. This implies $Pf(0)=Pf^\prime(0)=0$ and thus $f(0),f^\prime(0)
\in\mathcal{R}$.
\smalf
(ii) Simple calculations show for $v,\psi\in \mathcal{N}$ that
\begin{eqnarray}
\langle PL^\prime(0)v,\psi\rangle_{H^1_\per(Q_h)} & = &
\langle L^\prime(0)v,\psi\rangle_{H^1_\per(Q_h)}\ =\
\int\limits_{Q_h}\bigl[2[\tilde\theta\cdot\nabla_{\tilde{x}}v]\,
\overline{\psi}-2ik(q-\sin^2\theta_1)\,v\,\overline{\psi}\bigr]\,dx \nonumber \\
& & +\ 4\pi^2i\sum_{n\in \Z^2: |n+\alpha|>k}\frac{n\cdot\tilde\theta-k\cos^2\theta_1}
{\sqrt{|n+\alpha|^2-k^2}}\,\left(v^+_n\,\overline{\psi^+_n}+v^-_n\,\overline{\psi^-_n}
\right) \label{PK0}
\end{eqnarray}
where again $\alpha=k\tilde\theta=k\sin\theta_1\binom{\cos\theta_2}
{\sin\theta_2}$.
\smalf
It remains to justify the one-to-one property of the mapping $PL'(0)|_{\mathcal{N}}$
from $\mathcal{N}$ onto itself. First we observe that
$v,\psi\in \mathcal{N}$ can be extended into $Q_\infty\setminus Q_h$ by the Rayleigh
expansions
\begin{eqnarray*}
v(x) & = & \sum_{n\in \Z^2: |n+\alpha|>k}v_n^\pm\,
e^{\mp \sqrt{|n+\alpha|^2-k^2}(x_3\mp h)+i n\cdot \tilde{x}}\,,\quad \pm x_3>h\,, \\
\psi(x) & = & \sum_{n\in \Z^2: |n+\alpha|>k}\psi^\pm_n\,
e^{\mp \sqrt{|n+\alpha|^2-k^2}(x_3\mp h)+i n\cdot \tilde{x}},\,\quad \pm x_3>h\,.
\end{eqnarray*}
Now we compute the integral on the right hand of \eqref{PK0} but over the infinite
domain $U_h^+:=(Q_\infty\backslash Q_h)\cap \R^3_+$. Since $q\equiv 1$ in
$U_h^+$ we obtain
\begin{eqnarray*}
&&\int\limits_{U_h^+}\bigl[2[\tilde\theta\cdot\nabla_{\tilde{x}}v]\,
\overline{\psi}-2ik(q-\sin^2\theta_1)\,v\,\overline{\psi}\bigr]\,dx\ =\
2\int\limits_{U_h^+}\bigl[[\tilde\theta\cdot\nabla_{\tilde{x}}v]\,
\overline{\psi}-ik\cos^2\theta_1\,v\,\overline{\psi}\bigr]\,dx \\
& = & 8\pi^2 i\int\limits_h^\infty\sum_{n\in \Z^2: \,|n+\alpha|>k}v^+_n\,
\overline{\psi^+_n}\,[n\cdot\tilde\theta -k\cos^2\theta_1]\,
e^{-2\sqrt{|n+\alpha|^2-k^2}(x_3-h)}dx_3 \\
& = & 4\pi^2 i\sum_{n\in \Z^2: \,|n+\alpha|>k}v^+_n\,\overline{\psi^+_n}\,
\frac{n\cdot\tilde\theta-k\cos^2\theta_1}{\sqrt{|n+\alpha|^2-k^2}}\,.
\end{eqnarray*}
An analogous computation yields that the integral over $U^-_h:=
(Q_\infty\setminus Q_h)\cap \R^3_-$ gives
$$ \int\limits_{U^-_h}\bigl[2[\tilde\theta\cdot\nabla_{\tilde{x}}v]\,
\overline{\psi}-2ik(1-\sin^2\theta_1)\,v\,\overline{\psi}\bigr]\,dx\ =\
4\pi^2i\sum_{n\in \Z^2: |n+\alpha|>k}v^-_n\,\overline{\psi^-_n}\,
\frac{n\cdot\tilde\theta-k\cos^2\theta_1}{\sqrt{|n+\alpha|^2-k^2}}\,. $$
The previous two identities in combination with \eqref{PK0} yield
\begin{equation}\label{PK}
\langle PL^\prime(0)v,\psi\rangle_{H^1_\per(Q_h)}\ =\
2\int\limits_{Q_\infty}\bigl[[\tilde\theta\cdot\nabla_{\tilde{x}}v]\,\overline{\psi}-
ik(q-\sin^2\theta_1)\,v\,\overline{\psi}\bigr]\,dx
\end{equation}
for all $v,\psi\in\cN$.
\medlf
Assume now that $\langle PL^\prime(0)v,\cdot\rangle_{H^1_\per(Q_h)}$ vanishes
identically on $\cN$ for some $v\in\cN$. Then
$\langle PL^\prime(0)v,v\rangle_{H^1_\per(Q_h)}=0$ and thus
$$ \int\limits_{Q_\infty}[\tilde\theta\cdot\nabla_{\tilde{x}}v]\,\overline{v}\,dx\ =\
ik\int\limits_{Q_\infty}(q-\sin^2\theta_1)|v|^2dx\,, $$
i.e.
$$ \langle Bv,v\rangle_{H^1_\per(Q_h)}\ +\ 2k\,\langle Cv,v\rangle_{H^1_\per(Q_h)}\
=\ 0 $$
with the operators $B$ and $C$ from \eqref{eq:Op-B} and \eqref{eq:Op-C},
respectively. Since $v$ is evanescent and a solution of \eqref{eq:Hper} we conclude
by Green's theorem that also
$$ -\langle Av,v\rangle_{H^1_\per(Q_h)}\ +\ k\,\langle Bv,v\rangle_{H^1_\per(Q_h)}\
+\ k^2\langle Cv,v\rangle_{H^1_\per(Q_h)}\ =\ 0\,. $$
Eliminating the term with $B$ from the previous two equations yield
$\langle Av,v\rangle_{H^1_\per(Q_h)}+k^2\,\langle Cv,v\rangle_{H^1_\per(Q_h)}=0$ and
thus $v=0$ because $A$ is positive and $C$ is non-negative. This proves injectivity
of $PL^\prime(0)$ on $\cN$. \qed

\begin{remark} (i) If $q(x)\geq 1$ in $Q_h$, then the condition
$q(x)\geq\sin^2\theta_1$ in Lemma~\ref{lem:L-prop} (iii) and Lemma~\ref{lem:DL-prop} (ii)
is automatically satisfied.
\smalf
(ii) If $q(x)\leq\sin^2\theta_1$ on a part of $Q_h$ we do not expect the
existence of  {propagative wave vector}s. In the case of constant $q$, i.e. $q(x)=q_0$ with
$0<q_0<1$, there exist no  {propagative wave vector}s as it is well known. We repeat the
argument for the convenience of the reader. %\cblue{If $\alpha\in\R^2$ is a
 %{propagative wave vector} corresponding to $q\equiv q_0$ then its mode $\phi$ is evanescent.
%Since with $\phi$ also $\phi(\tilde x,x_3)+\phi(\tilde x,-x_3)$ is a mode we
%can assume that the mode is an even function with respect to $x_3$ and, therefore,
%has the form
%\ben
%\phi(x) & = & \sum_{|\alpha_n|>k} u_n\,e^{i\alpha_n\cdot\tilde x}\,
%e^{\mp|\beta_n|(x_3\mp 1)}\quad\mbox{for}\quad
%\pm x_3>1, \\
%\phi(x) & = &\sum_{|\alpha_n|>k} v_n e^{i\alpha_n\cdot\tilde x}
%\cos(\gamma_nx_3)\quad\mbox{for}\quad  |x_3|<1,
%\enn
%where $\gamma_n:=\sqrt{k^2q_0-|\alpha_n|^2}$ and $u_n, v_n\in \C$. Recall that
%$|\beta_n|=\sqrt{|\alpha_n|^2-k^2}$. Since $\phi$ and its derivative with respect to
%$x_3$ are continuous at $x_3=\pm 1$ we have the two equations
%$$ u_n-v_n\cos\gamma_n=0\quad\text{and}\quad
%-u_n\,|\beta_n|+v_n\,\gamma_n\sin\gamma_n=0\,, $$
%i.e. by solving for $v_n$,
%$$ v_n\bigl[|\beta_n|\cos\gamma_n-\gamma_n\sin\gamma_n\bigr]\ =\ 0\,. $$
%Now we use that $\gamma_n^2=k^2q_0-|\alpha_n|^2<k^2-|\alpha_n|^2<0$, i.e.
%$\gamma_n$ is purely imaginary and $|\beta_n|\cos\gamma_n-\gamma_n\sin\gamma_n=
%|\beta_n|\cosh|\gamma_n|+|\gamma_n|\sinh|\gamma_n|$ is real valued and
%positive. This implies $v_n=0$.}
%
%[$\phi(\tilde{x_3}, -x_3)$ is not a mode, because it does not satisfy the radiation condition !]
\cred{ We may expand the evanescent wave modes into the series
\ben
u(x)&=&\sum_{|\alpha_n|>k} u_n^\pm e^{i\alpha_n\cdot\tilde{x}\mp |\beta_n| (x_3\mp 1)}\quad\mbox{in}\quad \pm x_3>1, \\
u(x)&=&\sum_{|\alpha_n|>k} \left[v_n^+ e^{i\alpha_n\cdot\tilde{x}+ i\gamma_n x_3}+
v_n^- e^{i\alpha_n\cdot\tilde{x}-i\gamma_n x_3}\right]\quad\mbox{in}\quad  |x_3|<1,
\enn
where $\gamma_n:=\sqrt{k^2q_0-|\alpha_n|^2}$ and $u_n^\pm, v_n^\pm\in \C$. On $x_3=\pm 1$, the continuity of $u$ and $\partial_3 u$  yields the algebraic system
\ben
\begin{pmatrix}
	 1 & 0 & -e^{i\gamma_n} & -e^{-i\gamma_n} \\
	0 & 1  & -e^{-i\gamma_n} & -e^{i\gamma_n}\\
	-|\beta_n|  & 0 & -i\gamma_ne^{i\gamma_n} & i\gamma_ne^{-i\gamma_n} \\
	0 & |\beta_n|  & i\gamma_ne^{-i\gamma_n} & -i\gamma_ne^{i\gamma_n}
\end{pmatrix}
\begin{pmatrix}
	u_n^+ \\ u_n^- \\ v_n^+ \\ v_n^-
\end{pmatrix}=0,\quad \mbox{if}\quad |\alpha_n|<k.
\enn
 Since $q_0<1$, we have
$\gamma_n^2=k^2q_0-|\alpha_n|^2<k^2-|\alpha_n|^2<0$, i.e.,
$\gamma_n= i|\gamma_n|$ is purely imaginary.
Direct calculations gives the determinant of the coefficient matrix
\ben
(e^{-2|\gamma_n|}-e^{2|\gamma_n|})\, k^2(1-q_0)<0\quad\mbox{if}\quad |\alpha_n|<k.
\enn
Hence, the above linear system only admits the trivial solution $u_n^\pm=v_n^\pm=0$. This implies $u\equiv 0$, that is,  there is no evanescent wave modes if $q_0<1$.
}
\end{remark}

\subsection{Uniqueness and existence results}

Applying Lemma \ref{lem:sing-pert} we conclude that the unique solution $v(\epsilon)$
to \eqref{eq7} converges to $v$ in $X$ and the limiting function $v$ fulfils the
equations
\ben
L_k v\ =\ f_k\quad\mbox{and}\quad PL^\prime(0)v\ =\ 0\,.
\enn
The second equation provides an additional constraint on $v\in X$, that is (see
\eqref{PK}),
\begin{equation} \label{eq:orth-v}
\int\limits_{Q_\infty}\bigl[\tilde\theta\cdot\nabla_{\tilde{x}}v\,
\overline{\psi}\,dx\ -\ ik(q(x)-\sin^2\theta_1)\, v\,\overline{\psi}\bigr]\,dx\ =\ 0
\quad\text{for all }\psi\in \mathcal{N}=\mathcal{N}(L_k)
\end{equation}
where on the left hand side $v$ and $\psi$ are again extended to
$Q_\infty\setminus Q_h$ by their Rayleigh expansions.
Setting $u=e^{i\alpha\cdot \tilde{x}}v$ and $\phi=e^{i\alpha\cdot \tilde{x}}\psi$,
we return to the quasi-periodic setting to get
$$ \int\limits_{Q_\infty}[\tilde\theta\cdot\nabla_{\tilde{x}} u]\,\overline{\phi}\,dx\
=\ ik\int\limits_{Q_\infty}q(x)u\;\overline{\phi}\,dx\,\quad \mbox{for all }
\phi\in\mathcal{M}_k\,, $$
that is,
\begin{equation} \label{eq:orth-u}
\int\limits_{Q_\infty}\bigl[\tilde\theta\cdot\nabla_{\tilde{x}} u
-ikq(x) u\bigr]\,\overline{\phi}\,dx\ =\ 0\quad\mbox{for all }\phi\in\cM_k\,,
\end{equation}
where $\cM_k$ denotes the space of modes corresponding to the propagative wave
vector $\alpha=k\tilde\theta$, see \eqref{eq:mode}. We note that $\cM_k$
is finite dimensional.
\medlf
Well-posedness of acoustic scattering from bi-periodic inhomogeneous media is
summarized as follows.
\begin{theorem}\label{TH-LAP} Let $k>0$ be fixed and $\hat{\theta}\in \s^2_{-}$ be an
arbitrary incident direction. Set
$\alpha=k\tilde\theta=k\sin\theta_1(\cos\theta_2,\sin\theta_2)^\top$.
\begin{itemize}
\item[(i)] Then there exists a unique solution $u\in H^1_{\alpha,loc}(Q_\infty)$ such
that $u^{sc}:=u-u^{in}$ and $u^t$ satisfy the $\alpha$-quasi-periodic Rayleigh
expansions \eqref{eq:Rayleigh}, if $\alpha=k\tilde\theta$ is not a  {propagative wave vector}.
\item[(ii)] Suppose that $\alpha=k\tilde\theta$ is a  {propagative wave vector},
$q(x)\geq \sin^2\theta_1$ in $Q_h$ and that $|\alpha+n|\neq k$ for all
$n\in \Z^2$. Then the diffraction problem still admits a unique solution if the
total field is additionally required to fulfill the constraint condition
\be\label{or}
\int\limits_{Q_\infty}\bigl[\tilde\theta\cdot\nabla_{\tilde{x}}u\,-ikq\,u\bigr]\,
\overline{\phi}\,dx\ =\ 0\quad
\mbox{for all }\phi\in\mathcal{M}_k\,.
\en
\end{itemize}
\end{theorem}
\begin{proof} (i) By Lemma \ref{lem:L-prop}~(i) in combination with
Lemma~\ref{lem:equiv} the existence of $u\in H^1_{\alpha, loc}(Q_\infty)$ follows
from the Fredholm alternative and uniqueness holds true if $\alpha$ is not a propagative wave
vector.
\smalf
(ii) Let $\alpha$ be a  {propagative wave vector}. Existence of a solution is clear by the
limiting absorption principle shown abov (application of Lemma~\ref{lem:sing-pert}).
To show uniqueness we assume there are two solutions
$u^{(1)}$ and $u^{(2)}$ and set $w=u^{(1)}-u^{(2)}$ both satisfying \eqref{eq:orth-u}.
Then the restriction $v(x)=e^{-i\alpha\cdot\tilde{x}}w|_{Q_h}$ belongs to $\cN$ and
satisfies the relation \eqref{eq:orth-v}, that is,
$PL^\prime(0)v=0$. Applying Lemma \ref{lem:DL-prop} (ii) yields $v=0$ and thus
$w(x)=v(x) e^{i\alpha\cdot\tilde{x}}=0$.
\end{proof}
\begin{remark}(i) If $\alpha=k\tilde\theta$ is a  {propagative wave vector}, the general
solution to the diffraction problem is given as
\ben
u=\ u_0\ +\ \sum_{\ell=1}^m c_\ell\,\phi_{\ell}\quad\mbox{in}\;Q_\infty\,,\quad
c_\ell\in \C\,,
\enn
where $u_0\in H^1_{\alpha,loc}(Q_\infty)$ is a particular solution and
$\{\phi_{\ell}: \ell=1,2,\ldots,m\}\subset\mathcal{M}_k$ forms a
basis of the finite dimensional space $\mathcal{M}_k$ of modes. The constraint
\eqref{or} can be used to uniquely determine the coefficients $c_\ell$, and it does
not exclude guided waves.

(ii) Multiplying $k$ to the left hand side of \eqref{or}, we obtain an equivalent constraint condition as follows:
\ben
\int\limits_{Q_\infty}\bigl[\alpha\cdot\nabla_{\tilde{x}}u\,-ik^2q\,u\bigr]\,
\overline{\phi}\,dx\ =\ 0\quad
\mbox{for all }\phi\in\mathcal{M}_k\,.
\enn

\end{remark}

\section{Well posedness for an electromagnetic scattering problem}
\label{sec3}

\subsection{Problem formulation}

In this section we consider the half space
$\R_+^3=\{x\in\R^3 :  x_3>0\}$ instead of the full
space $\R^3$ and assume that $\R_+^3$ is filled with an inhomogeneous and
isotropic medium of electric permittivity  {$\epsilon=\epsilon(x)\geq\delta>0$,
magnetic permeability $\mu=\mu(x)\geq\delta>0$} and vanishing electric conductivity, which
sits on a perfectly conducting plate $\Gamma_0:=\{x\in\R^3: x_3=0\}$.
We assume that $\epsilon$ and $\mu$ are periodic with respect to $x_1$ and $x_2$ and,
without loss of generality, we take the periods to be $2\pi$ as in the scalar case treated
above. Furthermore, we assume that $\epsilon\equiv\epsilon_0>0$ and $\mu\equiv\mu_0>0$
 {for} $x_3>h$ for some $h>0$ and some constants $\epsilon_0,\mu_0>0$. The
electromagnetic wave propagation is governed by the time-harmonic Maxwell system (with the
time variation of the form $e^{-i\omega t}$ with frequency $\omega>0$)
\begin{equation}
\nabla\times\bE-i\omega\mu(x)\boldsymbol H=0,\quad\nabla\times\boldsymbol H+
i\omega \epsilon(x) \bE=0,
\end{equation}
where $\bE$ and $\boldsymbol H$ are the electric and magnetic fields,
respectively. In this section we denote the vectors in $\R^3$ in bold letters
(except of the variable $x\in\R^3$). Suppose that an electromagnetic plane wave
$(\bE^{in},\bH^{in})$ of the form
$$ \bE^{in}(x)=\bp\,e^{ikx\cdot\hat\btheta},\quad \bH^{in}(x)=
\bs\,e^{ikx\cdot \hat\btheta}, $$
is incident onto the inhomogeneous bi-periodic layer $Q:=\{x\in\R^3:
0<x_3<h\}$ from above. Here  $k=\omega\sqrt{\epsilon_0\mu_0}>0$ is the wavenumber
of the background homogeneous medium, $\hat\btheta:=(\sin\theta_1\cos\theta_2,
\sin\theta_1\sin\theta_2,-\cos\theta_1)^\top$ is the incident wave vector whose
direction is specified by $\theta_1\in(-\pi/2,\pi/2)$ and
$\theta_2\in(0,2\pi)$ and polarization vectors $\bp$ and $\bs$ satisfying
$\bp=\sqrt{\mu_0/\varepsilon_0}(\bs\times\hat\btheta)$ and $\bs\cdot\hat\btheta=0$.
Analogously to the acoustic case, we set $\alpha=k\tilde\theta$ with
$\tilde\theta:=\sin\theta_1\binom{\cos\theta_2}{\sin\theta_2}\in\R^2$. Obviously, the incident fields $\bE^{in}$ and
$\bH^{in}$ are \tcr{both} $\alpha$-quasi-periodic with respect to $\tilde{x}=(x_1,x_2)$,
that is,
$\bE^{in}(\tilde{x},x_{3})e^{-i\alpha\cdot\tilde{x}}$ and
$\boldsymbol H^{in}(\tilde{x},x_{3})e^{-i\alpha\cdot \tilde{x}}$ are
$2\pi$ periodic with respect to $x_1$ and $x_{2}$, respectively.
The diffraction of time-harmonic electromagnetic waves leads to the following
boundary value problem (here we eliminated the magnetic field from the system,
i.e. consider only the electric field $\bE$):

\begin{align}
\label{eq:odp-2}\nabla\times\bigl[\frac{\mu_0}{\mu(x)}\nabla\times\bE\bigr]-
k^2\, {\frac{\eps(x)}{\eps_0}}\bE=0 &\quad\text{in }\R_+^3, \\
\label{eq:odp-3}\bnu\times \bE=0 & \quad\text{for }x_2=0\,, \\
\label{eq:odp-4}\bE=\bE^{in}+\bE^{sc} & \quad\text{in }\R_+^3,
\end{align}

where
%$q(x)=\epsilon(x)/\epsilon_0$ is the refractive index,
$\bnu=(0,0,1)^\top$ is the unit normal at the boundary, and
$\bE^{sc}$ is the scattered electric field.

Similar to the Helmholtz equation, the above model still needs to be complemented
with an appropriate radiation condition of the scattered electric field
 {$\bE^{sc}$}. Here we assume that $\bE^{sc}$ fulfills the upward
Rayleigh expansion condition for the Maxwell system, i.e.
\be\label{eq:Rayleigh-Maxwell}
\bE^{sc}(x)=\sum_{n\in\Z^2}\bE_n
e^{i(\alpha_n\cdot\tilde{x}+\beta_nx_3)},\quad x_3>h,
\en
where again $\alpha_n=n+\alpha=n+k\tilde\theta\in\R^{2}$,
$\bE_{n}=(E_{n}^{(1)},E_{n}^{(2)},E_{n}^{(3)})^\top\in\mathbb{C}^{3}$
are constant vectors and $\beta_n=\beta_n(k)$ is given by \eqref{eqn:betaN},
i.e.
$$ \beta_n\ =\ \beta_n(k)\ :=\ \sqrt{k^2-|\alpha_n|^2}\ =\
\sqrt{k^2\cos^2\theta_1-2k\,n\cdot\tilde\theta-|n|^2}\,. $$
Throughout this section we keep $\tilde\theta$ fixed and assume again that
$\alpha=k\tilde\theta$ is not a cut-off vector, i.e.
\ben
\beta_n\neq 0\quad\mbox{for all}\quad n\in \Z^2.
\enn
From the Maxwell equations, the scattered electric is divergence free above the
inhomogeneous media, that is, $\ddiv\bE^{sc}=0$  {for} $x_3>h$. Hence the coefficients
$\bE_n$ satisfy
$$ \alpha_n\cdot\binom{E_n^{(1)}}{E_n^{(2)}}\ +\ \beta_n\,E_n^{(3)}\ =\ 0\quad
\mbox{for all }n\in \Z^2. $$

\begin{figure}[h]
\vspace*{-1.2cm}
\centering
\hspace{-2cm}\includegraphics[width=10cm,height=8cm]{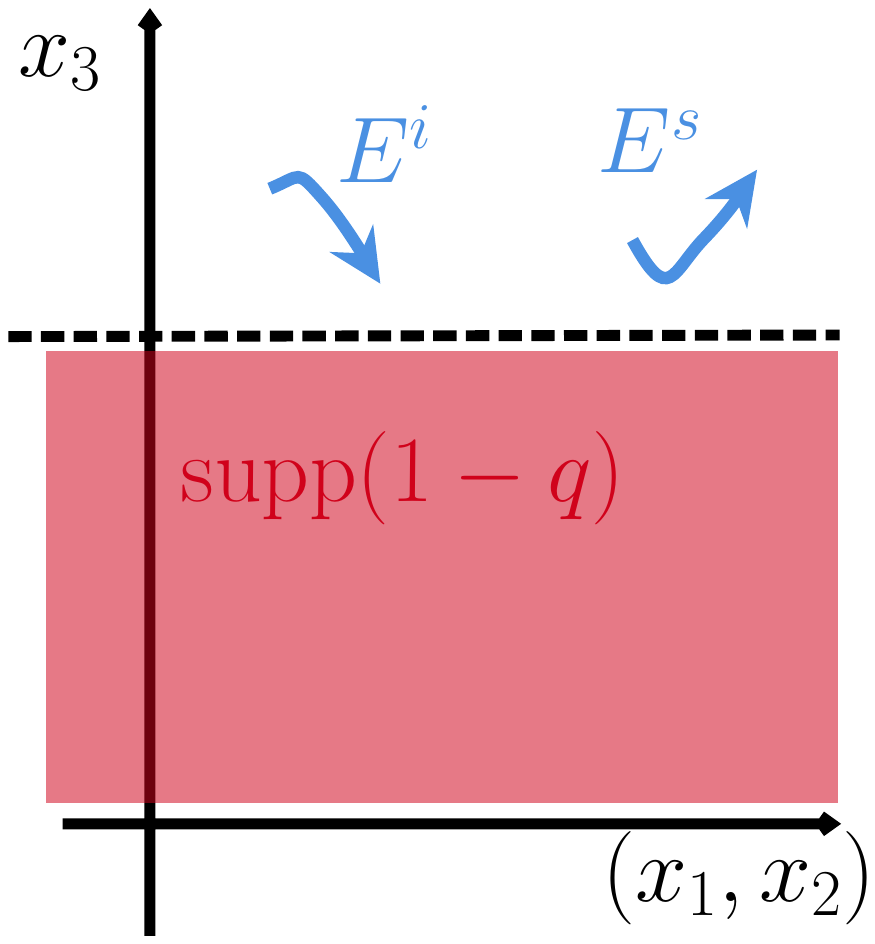}
\vspace*{-1.8cm}
\caption{Diffraction of an electromagnetic plane wave from a bi-periodic
inhomogeneous layer in $x_3>0$.}
\end{figure}

First, we recall the notation $\Gamma_h=\{(\tilde{x},h)\in \R^3 : 0<x_1,x_2<2\pi\}$
and modify the notations $Q_h$ and $Q_\infty$  {from Section~2} into
$$ Q_h\ :=\ (0,2\pi)^2\times(0,h)\quad\text{and}\quad Q_\infty\ :=\ (0,2\pi)^2
\times(0,\infty)\,, $$
respectively. The solution $\bE$ has to be in a Sobolev space accociated with the
quasi-periodic Maxwell system. As in the scalar case it it is sufficient to introduce
the spaces of periodic functions. For the corresponding $\alpha$-quasi-periodic spaces
(for any $\alpha\in\R^2$) the functions have to be multiplied by the phase factor
$e^{i\alpha\cdot\tilde x}$. For example,
$$ H_\alpha(\curl,Q_h)\ =\ \bigl\{e^{i\alpha\cdot\tilde x}\bu:\bu\in
H_\per(\curl,Q_h)\bigr\} $$
where $H_\per(\curl,Q_h)$ is the completion of
$$ C^\infty_\per(Q_h)^3\ =\ \bigl\{ \bu\in C^\infty(Q_h)^3:\bu
\text{ is $2\pi-$periodic with respect to $x_1$ and $x_2$} \bigr\} $$
with respect to the norm
$$ \Vert\bu\Vert_{H(\curl,Q_h)}\ :=\ \bigl[\Vert\bu\Vert^2_{L^2(Q_h)}+
\Vert\nabla\times\bu\Vert^2_{L^2(Q_h)}\bigr]^{1/2}\,, $$
see, e.g., \cite{GP22}. The corresponding local space $H_{\alpha,\loc}(\curl,Q_\infty)$
is defined as
$$ H_{\alpha,\loc}(\curl,Q_\infty)\ :=\ \bigl\{ \bu:\bu|_{Q_h}\in
H_\alpha(\curl,Q_h)\text{ for all }h>0 \bigr\}\,. $$
Since we have to include the boundary condition $\bnu\times\bu=0$ for $x_3=0$ we set
$H_{\alpha,0}(\curl,Q_h)\ :=\ \bigl\{ \bu\in H_\alpha(\curl,Q_h): \bnu\times\bu=0
\text{ for }x_3=0 \bigr\}$ and analogously, $H_{\alpha,\loc,0}(\curl,Q_\infty)$.
We note that the trace $\bnu\times\bu$ exists as an element in
$H^{-1/2}(\Div,\Gamma_0)$.
\smalf
Then $\bE\in H_{\alpha,\loc,0}(\curl,Q_\infty)$ is a (variational) solution of
\eqref{eq:odp-2}, \eqref{eq:odp-3} if
$$ \int\limits_{Q_\infty}\bigl[\frac{\mu_0}{\mu}(\nabla\times\bE)\cdot(\nabla\times\overline{\bpsi})-
k^2\, {\frac{\eps}{\eps_0}}\,\bE\cdot\overline{\bpsi}\bigr]dx\ =\ 0 $$
for all $\bpsi\in H_{\alpha,0}(\curl,Q_\infty)$ which vanish for $x_3>h$ for some $h>0$.

The boundary spaces $H^s_\alpha(\Gamma_h)$, $H^s_\alpha(\Div,\Gamma_h)$, and
$H^s_\alpha(\Curl,\Gamma_h)$ on $\Gamma_h$ are defined by
\begin{eqnarray*}
H^s_\alpha(\Gamma_h) & = & \biggl\{ \bu(\tilde{x})=\sum_{n\in\Z^2}
\bu_ne^{i\alpha_n\cdot\tilde{x}}:\bu_n\in\mathbb{C}^3\,,\ \bnu\cdot\bu_n=0,\
\|\bu\|_{H^s_\alpha(\Gamma_h)}<\infty\biggr\}\,, \\
H_\alpha^s(\Div,\Gamma_h) & = & \biggl\{\bu\in H^s_\alpha(\Gamma_h):
\|\bu\|_{H^s_\alpha(\Div,\Gamma_h)}<\infty\biggr\}\,, \\
H_\alpha^s(\Curl,\Gamma_h) & = & \biggl\{\bu\in H^s_\alpha(\Gamma_h):
\|\bu\|_{H^s_\alpha(\Curl,\Gamma_h)}<\infty\biggr\}\,,
\end{eqnarray*}
where the norms are defined as
\begin{eqnarray*}
\|\bu\|_{H^s(\Gamma_h)}^2 & := & \sum_{n\in\Z^2}(1+|n|^2)^s|\bu_n|^2\,, \\
\|\bu\|_{H^s(\Div,\Gamma_h)}^2 & := & \sum_{n\in\Z^2}
(1+|n|^2)^s(|\bu_n|^2+|\bu_n\cdot\bn|^2)\,, \\
\|\bu\|_{H^s(\Curl,\Gamma_h)}^2 & := & \sum_{n\in\Z^2}(1+|n|^2)^s
(|\bu_n|^2+|\bu_n\times\bn|^2)\,,
\end{eqnarray*}
with $\bn=(n_1,n_2,0)^\top\in\Z^3$. For $\alpha=(0,0)$ we denote the
spaces by $H^s_\per(\Gamma_h)$, $H^s_\per(\Div,\Gamma_h)$, and
$H^s_\per(\Curl,\Gamma_h)$, respectively.
We equip the spaces $H^s_\alpha(\Gamma_h)$, $H^s_\alpha(\Div,\Gamma_h)$, and
$H^s_\alpha(\Curl,\Gamma_h)$ with the same norms as for the spaces of periodic
functions.
\smalf
The case $s=-1/2$ is of particular importance because
$H_\alpha^{-1/2}(\Div,\Gamma_h)$ and  $H_\alpha^{-1/2}(\Curl,\Gamma_h)$ are the
trace spaces of $H_\alpha(\curl,Q_h)$, i.e. the trace operator
$\bu\mapsto\bnu\times\bu|_{\Gamma_h}$ is well defined and bounded from
$H_\alpha(\curl,Q_h)$ into $H_\alpha^{-1/2}(\Div,\Gamma_h)$. Since
$\bnu\times\bv\in H_\alpha^{-1/2}(\Curl,\Gamma_h)$ for $\bv\in
H_\alpha^{-1/2}(\Div,\Gamma_h)$ the trace operator
$\bu\mapsto(\bnu\times\bu)\times\bnu|_{\Gamma_h}$ is well defined and bounded from
$H_\alpha(\curl,Q_h)$ into $H_\alpha^{-1/2}(\Curl,\Gamma_h)$.
Furthermore, the spaces $H^{-1/2}_\alpha(\Div,\Gamma_h)$ and
$H^{-1/2}_\alpha(\Curl,\Gamma_h)$ are dual to each other and form a dual system
with dual form
$\langle\cdot,\cdot\rangle:H^{-1/2}_\alpha(\Curl,\Gamma_h)\times
H^{-1/2}_\alpha(\Div,\Gamma_h)\to\C$ which is the extension of
$$ \langle\bu,\bv\rangle\ =\ \sum_{n\in\Z^2}\bu_n\cdot\overline{\bv_n}\quad
\text{for }\bu,\bv\in H^1_\alpha(\Gamma_h)\,. $$
To reduce the scattering problem to a bounded periodic cell, we need to introduce
the quasi-periodic Calderon map (which is the analog of the Dirichlet-to-Neumann
map) $\widetilde\cT_\alpha:H_\alpha^{-1/2}(\Div,\Gamma_h)\to
H_\alpha^{-1/2}(\Curl,\Gamma_h)$ on the artificial boundary $\Gamma_h$  by
\begin{equation}\label{T}
\widetilde\cT_\alpha\bv\ =\
\bigl(\bnu\times(\nabla\times\bu)\bigr)\times\bnu\quad\text{ on }\Gamma_h\quad
\text{for }\bv\in H_\alpha^{-1/2}(\Div,\Gamma_h)
\end{equation}
where $\bu$ is the unique $\alpha$-quasi-periodic solution to the
boundary problem
\begin{equation}\label{TB} \nabla\times\nabla\times\bu-k^2\bu=0\quad\mathrm{for~}x_3>h,\quad\bnu\times
\bu=\bv\quad\text{ on }\Gamma_h\,,
\end{equation}
which also satisfies the upward Rayleigh expansion condition \eqref{eq:Rayleigh-Maxwell}. The map
$\widetilde\cT_\alpha$ is well defined if $\beta_n\neq 0$ for all $n\in \Z^2$. Below
we collect some properties of the Calderon map $\widetilde\cT_\alpha$ for general
$\alpha\in\R^2$. For notational convenience we set $\hat{\balpha}:=(\alpha,0)
\in \R^3$ and $\hat{\balpha}_n:=(\alpha_n, 0)\in \R^3$.
\begin{lemma} \label{Prop-Calderon}
Let $\bv(\tilde{x})=\sum_{n\in\Z^2}\bv_n\,e^{i\alpha_n\cdot\tilde{x}}\in
H_\alpha^{-1/2}(\Div,\Gamma_h)$ and suppose $|\alpha_n|\not=k$ for all $n\in\Z^2$.
\begin{itemize}
\item[(i)] $\widetilde\cT_\alpha:H_\alpha^{-1/2}(\Div,\Gamma_h)\to
H_\alpha^{-1/2}(\Curl,\Gamma_h)$ is well-defined and bounded and takes the explicit
form
\begin{equation*}
(\widetilde\cT_\alpha\bv)(\tilde{x})=i\sum_{n\in\Z^2}\frac1{\beta_n}
\bigl[k^2\bv_n-(\hat{\balpha}_n\cdot\bv_n)\hat{\balpha}_n]
e^{i\alpha_n\cdot\tilde{x}}.
\end{equation*}
\item[(ii)] The Calderon operator satisfies
\begin{eqnarray*}
\Re\langle\widetilde\cT_\alpha\bv,\bv\rangle & = & 4\pi^2\sum_{|\alpha_n|>k}
\frac1{|\beta_n|}\bigl[k^2|\bv_n|^2-|\hat{\balpha}_n\cdot\bv_n|^2
\bigr]\,, \\
\Im\langle\widetilde\cT_\alpha\bv,\bv\rangle & = & 4\pi^{2}\sum_{|\alpha_n|<k}
\frac{1}{\beta_n}\bigl[k^{2}|\bv_n|^2-|\hat{\balpha}_n\cdot\bv_n|^2]\
\geqslant\ 0\,,
\end{eqnarray*}
for $\bv\in H_\alpha^{-1/2}(\Div,\Gamma_h)$ where $C_1$ and $C_2$ are positive
constants, and \\ $\langle\cdot,\cdot\rangle:H_\alpha^{-1/2}(\Curl,\Gamma_h)\times
H_\alpha^{-1/2}(\Div,\Gamma_h)\to\C$ denotes the dual form.
\item[(iii)] For $p\in H^1_\alpha(Q_h)$ with $p(x)=\sum_{n\in\Z^2}p_n(x_3)\,
e^{i\alpha_n\cdot\tilde x}$ we have
$$ \Re\bigl\langle\widetilde\cT_\alpha(\bnu\times\nabla p),\bnu\times\nabla p
\bigr\rangle\ =\ 4\pi^2k^2\sum_{|\alpha_n|>k}\frac1{|\beta_n|}|\alpha_n|^2|p_n(h)|^2\
\geq\ 0\,. $$
\end{itemize}
\end{lemma}
\cred{In the first assertion, the expression of  $\widetilde\cT_\alpha$ can be derived straightforwardly from its definition \eqref{T}, which makes sense only if $|\alpha_n|\neq k$, that is, $\beta_n\neq 0$ for all $n \in \Z^2$. If  $|\alpha_n|= k$ for some $n\in \Z^2$, the Dirichlet boundary value problem \eqref{TB} over the half plane $x_3>h$  is not uniquely solvable and thus the operator  $\widetilde\cT_\alpha$ is not well defined. The real and imaginary parts of $\widetilde\cT_\alpha$ together with their properties follow directly from the explicit form of $\widetilde\cT_\alpha$. }
With the help of the Calderon map, the scattering problem can be equivalently
transformed into the following boundary value problem on the truncated domain $Q_h$.
Here, we allow general $\alpha\in\R^2$ but assume that the incident field $\bE^{in}$
is $\alpha$-quasi-periodic. If $\bE^{in}$ is a plane wave of the form
$\bE^{in}(x)=\bp\,e^{ikx\cdot\hat\btheta}$ then, of course, $\alpha$ has the form
$\alpha=k\tilde\theta$.

\begin{lemma} \label{lem:equiv-Maxwell}
Let $\bE^{in}$ be an $\alpha$-quasi-periodic solution of
\eqref{eq:odp-2} for \cred{ $\mu\equiv \mu_0$ and $\epsilon\equiv \epsilon_0$.}
\smalf
(a) Let $\bE\in H_{\alpha,\loc}(\curl,Q_\infty)$ be a solution of
\eqref{eq:odp-2}--\eqref{eq:odp-4} satisfying the Rayleigh expansion
\eqref{eq:Rayleigh-Maxwell}. Then $\bE|_{Q_h}\in H_\alpha(\curl,Q_h)$ solves
\begin{align}
\nabla\times\bigl[\frac{\mu_0}{\mu}\nabla\times\bE\bigr]-
k^2\, {\frac{\eps}{\eps_0}}\bE&=0\quad\text{in }Q_h\,,
\label{eq:dp-qp1} \\
\bnu\times\bE & =0\quad\text{on }\Gamma_0\,, \label{eq:vf-qp2} \\
(\nabla\times\bE)_T-\widetilde\cT_\alpha(\bnu\times\bE) & =
(\nabla\times\bE^{in})_T-\widetilde\cT_\alpha(\bnu\times\bE^{in})\quad\text{on }
\Gamma_h\,, \label{eq:dp-qp3}
\end{align}
where $\bu_T:=(\bnu\times\bu)\times\bnu$ denotes the trace on $\Gamma_h$ of the
tangential component of $\bu$.
\smalf
(b) Let $\bE\in H_\alpha(\curl,Q_h)$ \tcr{solve
(\ref{eq:dp-qp1})--(\ref{eq:dp-qp3})}. Extend $\bE$ to
$\bE_0=\bE^{in}+\bE_{0,sc}\in H_{\alpha,\loc}(\curl,Q_\infty)$ by setting
$\bE_{0,sc}=\bE^{sc}\tcr{=(\tilde\bE_{0,sc}, \bE^{(3)}_{0,sc})}$ in $Q_h$ with
\begin{equation} \label{eq:ext-Maxwell}
\begin{split}
\tilde\bE_{0,sc}(x)\ =\ & (\bE_{0,sc}^{(1)},\bE_{0,sc}^{(2)})^\top\ =\
\sum_{n\in\Z^2}\tilde\bE_n e^{i\alpha_n\cdot\tilde x+\beta_n(x_3-h))},\quad
\tcr{x_3>h}\,, \\
\bE_{0,sc}^{(3)}(x)\ =\ & -\sum_{n\in\Z^2}\frac{1}{\beta_n}\,
[n+k\tilde\theta]\cdot\tilde\bE_n\,
e^{i(\alpha_n\cdot\tilde x+\beta_n(x_3-h))},\quad \tcr{x_3>h}\,,
\end{split}
\end{equation}
where $\tilde\bE_n=\frac{1}{4\pi^2}\int_{\Gamma_{h}}\tilde\bE^{sc}\,
e^{-in\cdot\tilde x}ds$ are the Fourier coefficients of the tangential components of
$\bE^{sc}|_{\Gamma_{h}}=(\bE-\bE^{in})|_{\Gamma_{h}}$. Then
$\bE_0\in H_{\alpha,\loc}(\curl,Q_\infty)$ is a solution of
\eqref{eq:odp-2}--\eqref{eq:odp-4} satisfying the Rayleigh expansion
\eqref{eq:Rayleigh-Maxwell}.
\end{lemma}
\cred{The assertion (a) is a consequence of the definition of  the Calderon map $\widetilde\cT_\alpha$. In the assertion (b), the continuity of the tangential trace of $\bE_{0,sc}$ across $\Gamma_h$ guarantees the continuity of the tangential trace of $\nabla\times \bE_{0,sc}$, due to the boundary condition \eqref{eq:dp-qp3}.}
The variational formulation for the problem (\ref{eq:dp-qp1})--(\ref{eq:dp-qp3}) is
to find $\bE\in H_{\alpha,0}(\curl,Q_h)$ such that
\begin{equation}\label{qpv}
\tilde b_\alpha(\bE,\bpsi)\ =\ \tilde g_\alpha(\bpsi)\quad\mbox{for all }
\bpsi\in H_{\alpha,0}(\curl,Q_h)\,, \quad
\end{equation}
where
\begin{align*}
\tilde b_\alpha(\bE,\bpsi)&:=\int_{Q_h}\bigl[\frac{\mu_0}{\mu}(\nabla\times\bE)
\cdot(\nabla\times\overline{\bpsi})-k^2\, {\frac{\eps}{\eps_0}}\,
\bE\cdot\overline{\bpsi}\bigr]dx-
\int_{\Gamma_h}\widetilde\cT_\alpha(\bnu\times\bE)\cdot(\bnu\times\overline{\bpsi})
\,ds\,, \\
\tilde g_\alpha(\bpsi) &:=\int_{\Gamma_h}\bigl[(\nabla\times\bE^{in})_T-
\widetilde\cT_\alpha(\bnu\times\bE^{in})\bigr]\cdot(\bnu\times\overline{\bpsi})\,ds
\end{align*}
for $\bE,\bpsi\in H_{\alpha,0}(\curl,Q_h)$. Note that the integrals on $\Gamma_h$
are actually dual forms!
\begin{definition}
\tcr{Let $\alpha\in\R^2$ be arbitrarily fixed}. $\alpha\in\R^2$ is called a  {propagative wave vector} if the boundary value problem
(\ref{eq:odp-2})--(\ref{eq:Rayleigh-Maxwell}) with $\bE^{in}=0$
admits at least one non-trivial solution $\bE$ in $H_{\alpha,\loc}(\curl,Q_\infty)$
which are called modes. The space of all modes corresponding to the  {propagative wave vector}
$\alpha$ is denoted by $\cM_\alpha$.
\end{definition}

By Lemma~\ref{lem:equiv-Maxwell}, $\alpha$ is a  {propagative wave vector} if, and only if, the
homogeneous problem $\tilde b_\alpha(\bE,\bpsi)=0$ for all
$\bpsi\in H_{\alpha,0}(\curl,Q_h)$) admits at least one non-trivial solution in
$H_{\alpha,0}(\curl,Q_h)$.

With the representation theorem of Riesz we can write \eqref{qpv} in the form
\begin{equation} \label{eq:Lqpv}
\tilde L_\alpha\bE\ =\ \tilde\bff_\alpha
\end{equation}
where $\tilde L_\alpha:H_{\alpha,0}(\curl,Q_h)\to H_{\alpha,0}(\curl,Q_h)$ and
$\tilde\bff_\alpha\in H_{\alpha,0}(\curl,Q_h)$ are defined as
\begin{eqnarray}
\langle\tilde L_\alpha\bu,\bpsi\rangle_{H(\curl,Q_h)} & = &
\tilde b_\alpha(\bu,\psi) \nonumber \\
& = & \int\limits_{Q_h}\bigl[\frac{\mu_0}{\mu}(\nabla\times\bu)\cdot(\nabla\times\overline{\bpsi})-
k^2\cred{\frac{\epsilon}{\epsilon_0}}\,\bu\cdot\overline{\bpsi}\bigr]dx\\
\nonumber&&-\int\limits_{\Gamma_h}
\widetilde\cT_\alpha(\bnu\times\bu)\cdot(\bnu\times\overline{\bpsi})\,ds\,,
\label{eq:Ldef} \\
\langle\tilde\bff_\alpha,\bpsi\rangle_{H(\curl,Q_h)} & = &
\tilde g_\alpha(\bpsi)\ =\ \int\limits_{\Gamma_h}\bigl[(\nabla\times\bE^{in})_T-
\widetilde\cT_\alpha(\bnu\times\bE^{in})\bigr]\cdot(\bnu\times\overline{\bpsi})\,ds
\nonumber
\end{eqnarray}
where $\langle\cdot,\cdot\rangle_{H(\curl,Q_h)}$ denotes the inner product in
$H(\curl,Q_h)$ (which is the inner product in $H_{\alpha,0}(\curl,Q_h)$).

\begin{lemma} \label{lem:L-Fredholm-Maxwell}
Let $\alpha\in\R^2$ be arbitrary with $|\alpha_n|\not= k$ for all $n\in\Z^2$ where
$\alpha_n=n+\alpha$ for $n\in\Z^2$. Then the operator $\tilde L_\alpha$ is a Fredholm
operator with index zero from $H_{\alpha,0}(\curl,Q_h)$ into itself.
\end{lemma}
\begin{proof}
We use the quasi-periodic Hodge decomposition (or Helmholtz decomposition, see e.g.,
\cite{GP22}) in the form
\begin{equation} \label{eq:Hodge}
H_{\alpha,0}(\curl,Q_h)\ =\ \nabla H_{\alpha,0}^1(Q_h)\ \oplus\ X_{\alpha,0}
\end{equation}
where
\begin{eqnarray*}
H_{\alpha,0}^1(Q_h) & = & \left\{p\in H_\alpha^1(Q_h): p=0\mbox{ on }
\Gamma_0\right\}, \\
X_{\alpha,0} & = &\left\{\bu\in H_{\alpha,0}(\curl,Q_h):
\tilde b_\alpha(\bu,\nabla p)=0\mbox{ for all }p\in H_{\alpha,0}^1(Q_h)\right\}.
\end{eqnarray*}
In Lemma~\ref{lem:Hodge} of the appendix we show (again refering to \cite{GP22}) that
$X_{\alpha,0}$ is compactly embedded in $(L^2(Q_h))^3$.
With the Hodge decomposition \eqref{eq:Hodge} of $H_{\alpha,0}(\curl,Q_h)$ we can
represent $\tilde L_\alpha$ in matrix form (supressing the index $\alpha$) as
$$ \left[\begin{matrix} L_{11} & L_{12} \\ L_{21} & L_{22} \end{matrix}\right]:
X_{\alpha,0}\times\nabla H_{\alpha,0}^1(Q_h)\ \longrightarrow\
X_{\alpha,0}\times\nabla H_{\alpha,0}^1(Q_h)\,. $$
First we note that $L_{21}=0$. Indeed, $\langle L_{21}\bv,\nabla\varphi
\rangle_{H(\curl,Q_h)}=\tilde b_\alpha(\bv,\nabla\varphi)=0$ for $\bv\in X_{\alpha,0}$
and $\varphi\in H_{\alpha,0}^1(Q_h)$ by the definition of $X_{\alpha,0}$. We show that
$L_{11}$ is Fredholm and $L_{22}$ is an isomorphism from $\nabla H_{\alpha,0}^1(Q_h)$
onto itself. We begin with $L_{22}$ and note that $L_{22}\nabla p=\nabla\phi$ is
equivalent to
$$ k^2\int\limits_{Q_h}\cred{\frac{\epsilon}{\epsilon_0}}\,\nabla p\cdot\nabla\overline{\varphi}\,dx+
\int\limits_{\Gamma_h}\widetilde\cT_\alpha(\bnu\times\nabla p)\cdot
(\bnu\times\nabla\overline{\varphi})\,ds\ =\ -\langle\nabla\phi,\nabla\varphi
\rangle_{H(\curl,Q_h)} $$
for all $\varphi\in H^1_{\per,0}(Q_h)$. From part (iii) of Lemma~\ref{Prop-Calderon}
we conclude that the left hand side is coerive which proves that $L_{22}$ is an
isomorphism.
\newline
It remains to show that the part $L_{11}:X_{\alpha,0}\to X_{\alpha,0}$ is a Fredholm
operator. We decompose $\widetilde\cT_\alpha$ into $\widetilde\cT_\alpha=\cK_1-
\cK_2-\cC$ where
\begin{eqnarray*}
(\cK_1\bv)(\tilde x) & = & ik^2\sum_{n\in\Z^2}\frac{1}{\beta_n}\,\bv_n\,
e^{i\alpha_n\cdot\tilde x}\,, \\
(\cK_2\bv)(\tilde x) & = & i\sum_{|\alpha_n|<k}\frac{1}{\beta_n}\,
(\hat\balpha_n\cdot\bv_n)\,\hat\balpha_n\,e^{i\alpha_n\cdot\tilde x} \\
(\cC\bv)(\tilde x) & = & \sum_{|\alpha_n|>k}\frac{1}{|\beta_n|}\,
(\hat\balpha_n\cdot\bv_n)\,\hat\balpha_n\,e^{i\alpha_n\cdot\tilde x}\,.
\end{eqnarray*}
Then $\cC$ is obviously non-negative, i.e. $\langle\cC\bv,\bv\rangle\geq 0$ for all
$\bv\in H^{-1/2}_\alpha(\Gamma_h)$. We write $L_{11}$ in the form
\begin{eqnarray*}
\langle L_{11}\bu,\bpsi\rangle_{H(\curl,Q_h)} & = &
\int\limits_{Q_h}\bigl[\frac{\mu_0}{\mu}(\nabla\times\bu)\cdot(\nabla\times\overline{\bpsi})+
\bu\cdot\overline{\bpsi}\bigr]dx+\langle\cC(\bnu\times\bu),\bnu\times\bpsi
\rangle \\
& & -\ \int\limits_{Q_h}(1+k^2\cred{\frac{\epsilon}{\epsilon_0}})\,\bu\cdot\overline{\bpsi}\,dx\ -
\langle\cK_1(\bnu\times\bu),\bnu\times\bpsi\rangle \\
& & +\ \langle\cK_2(\bnu\times\bu),\bnu\times\bpsi\rangle\,.
\end{eqnarray*}
The first line represents a coercive operator from $X_{\alpha,0}$ onto itself,
\tcr{because $\mu(x)$ has a positive lower bound}. The
third line defines a compact operator because $\cK_2$ is finite dimensional. The
second line defines an operator $\cK$ from $X_{\alpha,0}$ into itself. We show that
$\cK$ is compact. We estimate for $\bu\in X_{\alpha,0}$
\begin{eqnarray*}
\bigl|\langle\cK\bu,\bpsi\rangle_{H(\curl,Q_h)}\bigr| & \leq &
c\,\Vert\bu\Vert_{L^2(Q_h)}\Vert\bpsi\Vert_{L^2(Q_h)}\\
& & +\ \Vert\cK_1(\bnu\times\bu)\Vert_{H^{-1/2}(\Curl,\Gamma_h)}
\Vert\bnu\times\bpsi\Vert_{H^{-1/2}(\Div,\Gamma_h)} \\
& \leq & c\,\bigl[\Vert\bu\Vert_{L^2(Q_h)}+
\Vert\cK_1(\bnu\times\bu)\Vert_{H^{-1/2}(\Curl,\Gamma_h)}\bigr]
\Vert\bpsi\Vert_{H(\curl,Q_h)}
\end{eqnarray*}
where we used the trace theorem and thus
\begin{equation} \label{eq:aux2}
\Vert\cK\bu\Vert_{H(\curl,Q_h)}\ \leq\ c\,\bigl[\Vert\bu\Vert_{L^2(Q_h)}+
\Vert\cK_1(\bnu\times\bu)\Vert_{H^{-1/2}(\Curl,\Gamma_h)}\bigr]\,.
\end{equation}
We compute, using $|\beta_n|\geq c\,[1+|n|^2]^{1/2}$,
\begin{equation} \label{eq:aux1}
\begin{split}
\Vert\cK_1(\bnu\times\bu)\Vert_{H^{-1/2}(\Curl,\Gamma_h)}^2\ =\ &
k^44\pi^2\sum_{n\in\Z^2}\frac{|\bnu\times\bu_n|^2+|(\bnu\times\bu_n)\times
\hat\balpha_n|^2}{|\beta_n|^2[1+|n|^2]^{1/2}} \\
=\ & 4\pi^2k^4\sum_{n\in\Z^2}\frac{|\bnu\times\bu_n|^2+|\bu_n\cdot\hat\balpha_n|^2}
{|\beta_n|^2[1+|n|^2]^{1/2}} \\
\leq\  & c\,\Vert\bnu\times\bu\Vert^2_{H^{-3/2}(\Gamma_h)}\ +\
4\pi^2k^4\sum_{n\in\Z^2}\frac{|\bu_n\cdot\hat\balpha_n|^2}{|\beta_n|^2[1+|n|^2]^{1/2}}
\end{split}
\end{equation}
because $\bnu\cdot\hat\balpha_n=0$. In the following we estimate the series on the right
hand side. We observe that $\widetilde\cT_\alpha(\bnu\times\bu)\times\bnu\in
H^{-1/2}(\Div,\Gamma_h)$. From the definition of $\widetilde\cT_\alpha$ we have
$$ \Div\bigl[\widetilde\cT_\alpha(\bnu\times\bu)\times\bnu\bigr]\ =\
-k^2\sum_{n\in\Z^2}\frac{1}{\beta_n}\bigl[(\bnu\times\bu_n)\times\bnu\bigr]\cdot
\hat\balpha_n\,e^{i\alpha_n\cdot\tilde x}\ =\
-k^2\sum_{n\in\Z^2}\frac{1}{\beta_n}\bu_n\cdot\hat\balpha_n\,
e^{i\alpha_n\cdot\tilde x} $$
and thus
$$ \bigl\Vert\Div\bigl[\widetilde\cT_\alpha(\bnu\times\bu)\times\bnu\bigr]
\bigr\Vert^2_{H^{-1/2}(\Gamma_h)}\ =\ 4\pi^2k^4\sum_{n\in\Z^2}
\frac{|\bu_n\cdot\hat\balpha_n|^2}{|\beta_n|^2[1+|n|^2]^{1/2}} $$
which \tcr{coincides} with the series on the right hand side of \eqref{eq:aux1}.
To compute the norm, let $\phi\in H^{1/2}(\Gamma_h)$ and extend $\phi$ to
$\phi\in H^1_{\alpha,0}(Q_h)$ with $\Vert\phi\Vert_{H^1(Q_h)}\leq
c\Vert\phi|{\gamma_h}\Vert_{H^{1/2}(\Gamma_h)}$ for some $c>0$ which is independent
of $\phi$. Then
\begin{eqnarray*}
\bigl\langle\Div\bigl[\widetilde\cT_\alpha(\bnu\times\bu)\times\bnu\bigr],\phi
\bigr\rangle & = &
-\bigl\langle\widetilde\cT_\alpha(\bnu\times\bu)\times\bnu,\Grad\phi\bigr\rangle \\
& = & \tilde b_\alpha(\bu,\nabla\phi)\ +\ k^2\int_{Q_h}\cred{\frac{\epsilon}{\epsilon_0}}\,\bu\cdot\nabla\phi\,dx\ =\
k^2\int_{Q_h}\cred{\frac{\epsilon}{\epsilon_0}}\,\bu\cdot\nabla\phi\,dx
\end{eqnarray*}
because $\bu\in X_{\alpha,0}$. Therefore,
$$ \bigl|\bigl\langle\Div\bigl[\widetilde\cT_\alpha(\bnu\times\bu)\times\bnu\bigr],
\phi\bigr\rangle\bigr|\ \leq\ c_1\Vert\bu\Vert_{L^2(Q_h)}\Vert\phi\Vert_{H^1(Q_h)}\
\leq\ c_2\Vert\bu\Vert_{L^2(Q_h)}\Vert\phi\Vert_{H^{1/2}(\Gamma_h)}\,. $$
Since this holds for all $\phi\in H^{1/2}(\Gamma_h)$ we conclude that
$$ \bigl\Vert\Div\bigl[\widetilde\cT_\alpha(\bnu\times\bu)\times\bnu\bigr]
\bigr\Vert_{H^{-1/2}(\Gamma_h)}\ \leq\ c\,\Vert\bu\Vert_{L^2(Q_h)}\,. $$
Substituting this into \eqref{eq:aux1} we obtain
$$ \Vert\cK_1(\bnu\times\bu)\Vert_{H^{-1/2}(\Curl,\Gamma_h)}^2\ \leq\
c\,\Vert\bnu\times\bu\Vert^2_{H^{-3/2}(\Gamma_h)}\ +\ c\,\Vert\bu\Vert^2_{L^2(Q_h)} $$
and thus by \eqref{eq:aux2}
\begin{equation*}
\Vert\cK\bu\Vert_{H(\curl,Q_h)}\ \leq\
c\,\bigl[\Vert\bnu\times\bu\Vert_{H^{-3/2}(\Gamma_h)}\ +\
\Vert\bu\Vert_{L^2(Q_h)}\bigr]\,.
\end{equation*}
This holds for all $\bu\in X_{\alpha,0}$ and proves compactness of $\cK$ as an operator from $X_{\alpha,0}$ into itself. Indeed, if $\bu^{(j)}\in X_{\alpha,0}$ converges weakly to
zero in the norm of $H_\alpha(\curl,Q_h)$ then $\bu^{(j)}$ converges to zero in the norm
of $L^2(Q_h)$ by the compact imbedding of $X_{\alpha,0}$ in $(L^2(Q_h))^3$.
Furthermore, the trace $\bnu\times\bu^{(j)}$ converges weakly to zero in
$H^{-1/2}(\Div,\Gamma_h)$ by the trace theorem and thus in the norm of
$H^{-3/2}(\Gamma)$ because $H^{-1/2}(\Div,\Gamma_h)$ is compactly embedded in
$H^{-3/2}(\Gamma)$. This shows that $L_{11}$ is a Fredholm operator and ends the proof.
\end{proof}

\begin{lemma}\label{lem:L-prop-Maxwell}
Let again $\alpha\in\R^2$ be arbitrary with $|\alpha_n|\not= k$ for all $n\in\Z^2$.
\begin{itemize}
\item[(i)] The operator $\tilde L_\alpha$ is an isomorphism from
$H_{\alpha,0}(\curl,Q_h)$ onto itself if $\alpha$ is not a  {propagative wave vector}.
If $\alpha$ is a  {propagative wave vector}, then the null space $\cN(\tilde L_\alpha)$
concides with the nullspace $\cN(\tilde L_\alpha^\ast)$ of the adjoint, is finite dimensional,
and consists of surface wave modes only, i.e. the extension $\bE_0$ of
$\bE\in\cN(\tilde L_\alpha)$ to $Q_\infty$, given by \eqref{eq:ext-Maxwell}, takes
the form
\begin{equation}\label{eq12 Maxwell}
\bE_0(x)\ =\ \sum_{ n\in\Z^2:|n+\alpha|>k}\bE_n\,e^{i\alpha_n\cdot\tilde{x}-
|\beta_n|(x_3-h)}\,,\quad  x_3>h\,,
\end{equation}
where $\bE_n\in \C^3$ satisfies $\hat{\balpha}_n\cdot\bE_n=0$ for all $n\in \Z^2$.
\item[(ii)] The Riesz number of $\tilde L_\alpha$ is one, that is,
$\cN(\tilde L_\alpha)=\cN(\tilde L_\alpha^2)$. Moreover,  {the decomposition
$H_{\alpha,0}(\curl,Q_h)=\cN(\tilde L_\alpha)\oplus\cR(\tilde L_\alpha)$ is orthogonal
with respect to $\langle\cdot,\cdot\rangle_{H(\curl,Q_h)}$}.
Here $\cR(\tilde L_\alpha)$ denotes the range of the operator $\tilde L_\alpha$.
\item[(iii)] If $\bE^{in}(x)=\bp\,e^{ik\hat\btheta\cdot x}$ with $\bp\cdot\hat\btheta=0$
and $\alpha$ is of the particular form $\alpha=k\tilde\theta=
k\sin\theta_1\binom{\cos\theta_2}{\sin\theta_2}$, then the equation \eqref{eq:Lqpv}
admits at least one solution $\bE\in H_{\alpha,0}(\curl,Q_h)$.
\end{itemize}
\end{lemma}
\begin{proof}
(i) The first assertion is obvious because $\tilde L_\alpha$ is Fredholm. Let
$\bE\in\cN(\tilde L_\alpha)$. Then $\tilde b_\alpha(\bE,\bpsi)$ for all $\bpsi$ and
thus, taking $\bpsi=\bE$ and the imaginary part we obtain
$\langle\widetilde\cT_\alpha(\bnu\times\bE),\bnu\times\bE\rangle=0$. By part (ii) of
Lemma~\ref{Prop-Calderon} we obtain
$$ \sum_{|\alpha_n|<k}\frac{1}{\beta_n}\bigl[k^{2}|\bE_n|^2-
|\hat{\balpha}_n\cdot\bE_n|^2]\ =\ 0 $$ and thus
$$ 0\ =\ \sum_{|\alpha_n|<k}\frac{1}{\beta_{n}}[k^{2}|\bE_{n}|^{2}-
|\hat\balpha_{n}\cdot\bE_{n}|^2]\ \geq\ \sum_{|\alpha_n|<k}\frac{1}{\beta_{n}}
(k^{2}-|\alpha_{n}|^2)|\bE_{n}|^{2} $$
from which $\bE_n=0$ follows for all $n\in\Z^2$ with $|\alpha_n|<k$. This proves
\eqref{eq12 Maxwell}.
\smalf
Using the above property it is easy to verify that
$\tilde b_\alpha(\bu,\bv)=\overline{\tilde b_\alpha(\bv,\bu)}$ for all
$\bu\in\cN(\tilde L_\alpha)$ and $\bv\in H_{\alpha,0}(\curl,Q_h)$.
Then we have for $\bu\in\cN(\tilde L_\alpha)$ and arbitrary
$\bv\in  H_{\alpha,0}(\curl,Q_h)$ that
$$ 0\ =\ \langle\tilde L_\alpha\bu,\bv\rangle_{H(\curl,Q_h)}\ =\
\overline{\langle \tilde L_\alpha\bv,\bu\rangle_{H(\curl,Q_h)}}\ =\
\langle\tilde L_\alpha^\ast\bu,\bv\rangle_{H(\curl,Q_h)}\,, $$
i.e. $\tilde L_\alpha^\ast\bu=0$ and thus $\cN(\tilde L_\alpha)\subset
\cN(\tilde L_\alpha^\ast)$. The reverse inclusion is proven in the same way.
\smalf
(ii) It sufficies to prove  $\cN(\tilde L_\alpha^2) \subset\cN(\tilde L_\alpha)$.
Given $\bu\in H_{\alpha,0}(\curl,Q_h)$ with $\tilde L_\alpha^2\bu=0$, we set
$\bv=\tilde L_\alpha\bu\in\cR(\tilde L_\alpha).$ Since $\bv\in\cN(\tilde L_\alpha)=
\cN(\tilde L_\alpha^\ast)$, we obtain
$$ \|\bv\|_{H(\curl,Q_h)}^2\ =\ \langle\bv,\bv\rangle_{H(\curl,Q_h)}\ =\
\langle\bv,\tilde L_\alpha\bu\rangle_{H(\curl,Q_h)}\ =\
\langle\tilde L_\alpha^\ast\bv,\bu\rangle_{H(\curl,Q_h)}\ =\ 0\,, $$
which proves $\bv=\tilde L_\alpha\bu=0$, i.e.
$\cN(\tilde L_\alpha^2)\subset\cN(\tilde L_\alpha)$.
\newline
Therefore, the Riesz number is one which implies the decomposition
$H_{\alpha,0}(\curl,Q_h)=\cN(\tilde L_\alpha)\oplus\cR(\tilde L_\alpha)$.
The orthogonality follows from the relation
$\cN(\tilde L_\alpha)=\cN(\tilde L_\alpha^\ast)$.
\smalf
(iii) By the Fredholm alternative existence of a solution
$\bE\in H_{\alpha,0}(\curl,Q_h)$ of \eqref{eq:Lqpv} is assured if the right hand side
$\tilde\bff_\alpha$, defined in \eqref{eq:Ldef}, is orthogonal to the nullspace of
$\tilde L_\alpha^\ast$ which coincides with the nullspace of $\tilde L_\alpha$, i.e.
it remains to show that $\langle\tilde\bff_\alpha,\bpsi\rangle_{H(\curl,Q_h)}=0$ for
all $\bpsi\in\cN(\tilde L_\alpha)$. From part (i) we conclude that the coefficient
$\bpsi_n$ for $n=(0,0)$ vanishes, i.e. $\int_{\Gamma_h}\bpsi(\tilde x,h)\,
e^{-ik\tilde\theta\cdot\tilde x}d\tilde x=0$. From the definition of $\bE^{in}$ we
have the representation
$$ \bigl[(\nabla\times\bE^{in})_T-\widetilde\cT_\alpha(\bnu\times\bE^{in})
\bigr](\tilde x)\ =\ \bq(k)\,e^{ik\tilde\theta\cdot\tilde x} $$
with
\begin{equation} \label{eq:q_k}
\bq(k)\ =\ \biggl[ik\,\bnu\times(\hat\theta\times\bp)\times\bnu-
\frac{ik}{\cos\theta_1}(\bnu\times\bp)+\frac{ik}{\cos\theta_1}
(\check\btheta\cdot(\bnu\times\bp))\check\btheta\biggr]\,e^{-ikh\cos\theta_1}
\end{equation}
where we have set $\check\btheta=(\tilde\theta,0)^\top\in\R^3$ and thus
$$ \langle\tilde\bff_\alpha,\bpsi\rangle_{H(\curl,Q_h)} = \int_{\Gamma_h}
\bq(k)\cdot(\bnu\times\overline{\bpsi(\tilde x,h)})\,e^{ik\tilde\theta\cdot\tilde x}
d\tilde x = (\bq(k)\times\bnu)\cdot\int_{\Gamma_h}\overline{\bpsi(\tilde x,h)\,
e^{-ik\tilde\theta\cdot\tilde x}}d\tilde x = 0\,. $$
\end{proof}

% {The following blue paragraph has to be modified and placed to where?}

%\cblue{
%Using the Hodge decomposition, one can show that for given incident direction
%$\hat\theta$ the $\alpha-$quasi-periodic boundary value problem
%(\ref{eq:dp-qp1})--(\ref{eq:dp-qp3}) for $\alpha=k\tilde\theta$ is uniquely
%solvable in $H_{\alpha,0}(\curl,Q_h)$ if $k\in\R_+$ does not belong to a discrete
%set $\mathcal{D}\subset \R_+$ with the only accumulating point at infinity. The
%exceptional set $\mathcal{D}$ consists of two parts: $\{k\in \R_+: |n+\alpha|=k
%\mbox{ for some }n\in \Z^2\}$ (that is, $\alpha=k\tilde\theta$ is  a cut-off
%vector)  and when $\alpha=k\tilde\theta$ is a  {propagative wave vector}. The first part is
%caused by Rayleigh frequencies (so that the Calderon map $\widetilde\cT_\alpha$ is
%not well-defined), while the second part does not ensure uniqueness\footnote{The
%first part is not essential. In fact, well-posedness of the diffraction problem
%(\ref{eq:odp-2})--(\ref{eq:Rayleigh-Maxwell}) can be justified for non-critical
%vectors $\alpha$ even if $|n+\alpha|=k$ for some $n\in\Z^2$. We refer to \cite{HR14,
%HR15} for the variational approach coupled with a truncating technique based on the
%Fourier mode expansion instead of the Calderon map.}.
%Moreover, uniqueness and existence hold true for all $k>0$ if the material in
%$Q_{h_0}$ is absorbing, that is, $\Im q\geq q_0$ on $Q_{h_0}$ for some $q_0>0$.
%\footnote{This can be relaxed if $q$ is sufficiently smooth such that the unique
%continuation principle holds.}
%}

By Lemma \ref{lem:L-prop-Maxwell} (i), one can show that for given incident direction
$\hat\theta$ the $\alpha$-quasi-periodic boundary value problem
(\ref{eq:dp-qp1})--(\ref{eq:dp-qp3}) with $\alpha=k\tilde\theta$ is uniquely
solvable in $H_{\alpha,0}(\curl,Q_h)$, provided $\alpha$ is neither a  {propagative wave
vector} nor a cut-off vector. Moreover, uniqueness and existence hold true for all $k>0$
and $\alpha\in \R^2$ if the material in $Q_{h}$ is absorbing, that is, $\Im q\geq q_0$ on
$Q_{h}$ for some $q_0>0$. \footnote{This can be relaxed if  {$\mu$ and $\eps$ are}
sufficiently smooth such
that the unique continuation principle holds.}  If $q$ is real-valued (non-absorbing), there
is, in general, no uniqueness to our scattering problem due to the possible existence of
guided waves. The aim of this section is to justify the uniqueness and existence of
quasi-periodic solutions when
$\alpha=k\tilde\theta=k\sin\theta_1(\cos\theta_2,\sin\theta_2)^\top$ is a
 {propagative wave vector} vector but still not a cut-off vector. As in the Helmholtz
case, we shall carry over the limiting absorbing arguments to the Maxwell equations by
investigating the convergence of the solutions when the imaginary part of
$k\in \C_+$ tends to zero.

\subsection{Transform to space of periodic functions}

In this section we keep $\tilde\theta$ and $k>0$ fixed and set $\alpha=k\tilde\theta$.
As in the scalar case we transform the \tcr{$\alpha$-dependent} setting to the periodic
setting by replacing $\bE(x)$ by $\bE(x)=e^{ik\tilde\theta\cdot\tilde x}\bF(x)=
e^{ik\check\btheta\cdot x}\bF(x)$ for a field $\bF$ which is periodic with respect
to $\tilde x$. Here, we recall $\check\btheta=(\tilde\theta,0)\in\R^3$. Using
$\nabla\times[e^{ik\,\check\btheta\cdot x}\bF(x)]=[ik\check\btheta\times\bF(x)+
\nabla\times\bF(x)]e^{ik\check\btheta\cdot x}$, we note that the scattering problem
\eqref{eq:odp-2}--\eqref{eq:Rayleigh-Maxwell} turns into
 {\begin{equation} \label{eq:Maxwell-per}
\begin{split}
& \nabla\times\bigl[\frac{\mu_0}{\mu}(\nabla\times\bF)\bigr]
+ik\,\bigl\{\check\btheta\times\bigl[\frac{\mu_0}{\mu}(\nabla\times\bF)\bigr]+
\nabla\times\bigl[\frac{\mu_0}{\mu}(\check\btheta\times\bF)\bigr]\bigr\} \\
& -k^2\bigl\{\frac{\eps}{\eps_0}\bF-\check\btheta\times\bigl[\frac{\mu_0}{\mu}
(\bF\times\check\btheta)\bigr]\bigr\}\ =\ 0 \text{ in }\R_+^3,
\end{split}
\end{equation}}
\begin{equation} \label{eq:bc-per}
\bnu\times\bF=0\quad\text{for}\,\;\cred{x_3=0}\,,\qquad
\bF=\bF^{in}+\bF^{sc}\quad\text{in }\R_+^3,
\end{equation}
where $\bF^{in}(x)=\bp\,e^{-ik\cos\theta_1\,x_3}$, and $\bF^{sc}$ satisfies the
Rayleigh expansion \eqref{eq:Rayleigh-Maxwell} with $\alpha_n$ replaced by $n$, i.e.
\begin{equation} \label{eq:Rayleigh-per-Maxwell}
\bF^{sc}(x)\ =\ \sum_{n\in\Z^2}\bF_n
e^{i(n\cdot\tilde{x}+\beta_nx_3)},\quad x_3>h\,.
\end{equation}Note that $\ddiv\bigl[e^{ik\tilde\theta\cdot\tilde x}\bF\bigr]=0$
is equivalent to
\begin{equation} \label{eq:divF}
\beta_nF_n^{(3)}\ +\ (k\tilde\theta+n)\cdot\cred{\tilde F_n}\ =\ 0\quad\text{for all }
n\in\Z^2
\end{equation}
where $\cred{\tilde F_n}=(F_n^{(1)},F_n^{(2)})^\top$ which expresses $F_n^{(3)}$
by $F_n^{(1)}$ and $F_n^{(2)}$.
%\cblue{Question concerning notation: We denoted vectors in $\R^3$ by bold letters in
%contrast to vectors in $\R^2$ (eg. $\check\btheta\in\R^3$ and
%$\tilde\theta\in\R^2$). Should we be consisting with this? Then we have to replace
%$\tilde\bF_n$ by $\tilde F_n$ and so on.}

To reduce the corresponding variational form to the bounded cell
$Q_h=(0,2\pi)^2\times(0,h)$ we define the periodic Calderon map
$\cT_k:H_\per^{-1/2}(\Div,\Gamma_h)\to H_\per^{-1/2}(\Curl,\Gamma_h)$ by
\begin{equation}
(\cT_k\bv)(\tilde x)\ :=\ e^{-i k\tilde\theta\cdot\tilde x}
(\widetilde\cT_{k\tilde\theta}(\bv\,e^{ik\tilde\theta\cdot\tilde x}))(\tilde x)\quad
\text{for }\bv\in H_\per^{-1/2}(\Div,\Gamma_h)\,.
\end{equation}
Setting $\bE=e^{ik\tilde\theta\cdot\tilde{x}}\bF$ and replacing $\bpsi$ by
$e^{ik\tilde\theta\cdot\tilde{x}}\bpsi$ in \eqref{qpv} for $\bF,\bpsi\in
H_{\per,0}(\curl,Q_h)$ we obtain the periodic variational form of
\eqref{qpv} as
\begin{equation} \label{eq:var-per}
b_k(\bF,\bpsi)\ =\ g_k(\bpsi)\quad\mbox{for all }\bpsi\in H_{\per,0}(\curl,Q_h)\,,
\end{equation}
where
\begin{align}
b_k(\bv,\bpsi)\quad = &\quad \tilde b_{k\tilde\theta}(e^{ik\tilde\theta\cdot\tilde{x}}\bv,
e^{ik\tilde\theta\cdot\tilde{x}}\bpsi) \nonumber \\
= &\quad \int\limits_{Q_h}\bigl[\frac{\mu_0}{\mu}\nabla\times
\bigl(e^{ik\tilde\theta\cdot\tilde{x}}\bv\bigr)\bigr]
\cdot\bigl[\nabla\times\bigl(\overline{e^{ik\tilde\theta\cdot\tilde{x}}\bpsi}\bigr)
\bigr]-k^2\cred{\frac{\epsilon}{\epsilon_0}}\,\bv\cdot\overline\bpsi\,dx \nonumber \\
& -\int\limits_{\Gamma_h}\cT_k(\bnu\times\bv)\cdot(\bnu\times\overline\bpsi)\,ds \nonumber \\
= &\quad \int\limits_{Q_h}\frac{\mu_0}{\mu}\bigl[(\nabla\times\bv)\cdot
(\nabla\times\overline\bpsi)
+ik\,(\nabla\times\overline\bpsi)\cdot(\check\btheta\times\bv)-
ik\,(\nabla\times\bv)\cdot(\check\btheta\times\overline\bpsi)\bigr] \label{eq:b_alpha} \\
& -\ k^2\bigl\{ {\frac{\eps}{\eps_0}}\,\bv-\bigl[\frac{\mu_0}{\mu}\sin^2\theta_1\bv+
(\check\btheta\cdot\bv)\check\btheta\bigr]\bigr\}\cdot
\overline\bpsi\,dx\ -\int\limits_{\Gamma_h}
\cT_k(\bnu\times\bv)\cdot(\bnu\times\overline\bpsi)\,ds\,, \nonumber
\end{align}
\begin{equation} \label{eq:g_alpha}
\begin{split}
g_k(\bpsi) = &\ \tilde g_{k\tilde\theta}(e^{ik\tilde\theta\cdot\tilde x}\bpsi)\ =\
\int\limits_{\Gamma_h}\bigl[(\nabla\times\bE^{in})_T-
 {\cT_k(\bnu\times\bE^{in})}\bigr]\cdot(\bnu\times\overline\bpsi)\,
e^{-ik\tilde\theta\cdot\tilde x}\,ds \\
= &\ \bq(k)\cdot\int\limits_{\Gamma_h}\bnu\times\overline\bpsi\,ds
\end{split}
\end{equation}
with $\bq(k)$ from \eqref{eq:q_k}. By the representation theorem of Riesz the equation
\eqref{eq:var-per} is equivalent to
\begin{equation} \label{eq:Lv-per}
L_k\bF\ =\ \bff_k
\end{equation}
for $\bF\in H_{\per,0}(\curl,Q_h)$, where
$\langle L_k\bv,\bpsi\rangle_{H(\curl,Q_h)}=b_k(\bv,\bpsi)$ and
$\langle\bff_k,\bpsi\rangle_{H(\curl,Q_h)}=g_k(\bpsi)$ for
$\bv,\bpsi\in H_{\per,0}(\curl,Q_h)$.
We note that $L_k:H_{\per,0}(\curl,Q_h)\to H_{\per,0}(\curl,Q_h)$ satisfies all of
properties of $\tilde L_\alpha$ stated in Lemmas~\ref{lem:L-prop-Maxwell}. Indeed, if
we define the isomorphism $J_\alpha:H_{\per,0}(\curl,Q_h)\to H_{\alpha,0}(\curl,Q_h)$
by $J_\alpha\bv(x)=e^{i\alpha\cdot\tilde x}\bv(x)$ then
\begin{equation*}
\langle L_k\bv,\bpsi\rangle_{H(\curl,Q_h)}\ =\
\tilde b_\alpha(J_\alpha\bv,J_\alpha\bpsi)\ =\
\langle\tilde L_\alpha J_\alpha\bv,J_\alpha\bpsi\rangle_{H(\curl,Q_h)}\ =\
\langle J_\alpha^\ast\tilde L_\alpha J_\alpha\bv,\bpsi\rangle_{H(\curl,Q_h)}
\end{equation*}
for $\alpha=k\tilde\theta$, i.e. $L_k=J_\alpha^\ast\tilde L_\alpha J_\alpha$.
In particular, $L_k$ is a Fredholm operator with index zero and Riesz number one.

\subsection{The limiting absorption argument}

We consider again the case that $\alpha=k\tilde\theta$, i.e. $\hat\balpha=k\check\btheta$
where $\check\btheta:=(\tilde\theta,0)^\top=
\sin\theta_1(\cos\theta_2,\sin\theta_2,0)^\top\in\R^3$. By the definition of $b_k$ we
note that we can extend the definition of $L_k$ to the case of complex $k$ with
$\Im k>0$ (since $\beta_n=\beta_n(k)$ is defined in \eqref{eqn:betaN}, see
Lemma~\ref{lem:beta}). Then the following analog to part (iii) of Lemma~\ref{lem:L-prop}
holds.
\begin{lemma} \label{lem:L-isomorphism}
Suppose that  {$\eps(x)\mu(x)\geq\eps_0\mu_0\sin^2\theta_1$ in $Q_\infty$.} For every
$k_0>0$ there exists $ {\kappa_0}>0$ such that $L_k$ is an isomorphism from
$H_{\per,0}(\curl,Q_h)$ onto itself for all $k\in\C$ with $|k-k_0|< {\kappa_0}$ and
$\Im k>0$. Furthermore, there is a unique solution $\bF\in H_{\per,loc}(\curl,\R^3)$ to
\eqref{eq:Maxwell-per}.
\end{lemma}
\begin{proof}
First we note that the operator $L_k$ is Fredholm for real values of $k$. Since the
set of Fredholm operators is open in the space of bounded operaturs we conclude that
also $L_k$ is Fredholm for sufficiently small $\Im k>0$. Therefore, it suffices to show
injectivity. Let $\bv\in H_{\per,0}(\curl,Q_h)$ with $L_k\bv=0$, i.e.
$b_k(\bv,\bpsi)=0$ for all $\bpsi\in H_{\per,0}(\curl,Q_h)$. For $\alpha=k\tilde\theta$,
i.e. $\hat\balpha=k\check\btheta$, the above form of $b_k$ takes the form
\begin{align*}
b_k(\bv,\bpsi)\ & =\ \int\limits_{Q_h}\frac{\mu_0}{\mu}[(\nabla\times\bv)\cdot(\nabla\times\overline\bpsi)
+ik\,(\nabla\times\overline\bpsi)\cdot(\check\btheta\times\bv)-
ik\,(\nabla\times\bv)\cdot(\check\btheta\times\overline\bpsi)] \\
& \quad -k^2\bigl( {\frac{\eps}{\eps_0}}-\frac{\mu_0}{\mu}\sin^2\theta_1\bigr)\,\bv\cdot\overline\bpsi-k^2\frac{\mu_0}{\mu}(\check\btheta\cdot\bv)
(\check\btheta\cdot\overline\bpsi)\,dx
-\int\limits_{\Gamma_h}\cT_\alpha(\bnu\times\bv)\cdot(\bnu\times\overline\bpsi)\,ds\,.
\end{align*}
We extend $\bv$ by the Rayleigh expansion to $x_3>h$, i.e. we set \cred{$\bv=(\tilde{v}, v^{(3)})$ with}
\begin{eqnarray*}
\cred{\tilde v}(x) & = & (v^{(1)},v^{(2)})^\top\ =\ \sum_{n\in\Z^2}\tilde v_n
e^{i(n\cdot\tilde x+\beta_n(x_3-h))},\quad x_3>h\,, \\
\cred{v^{(3)}}(x) & = & -\sum_{n\in\Z^2}\frac{1}{\beta_n}\,[n+k\tilde\theta]
\cdot\tilde v_n\,e^{i(n\cdot\tilde x+\beta_n(x_3-h))},\quad x_3>h\,,
\end{eqnarray*}
where $\tilde v_n=\frac{1}{4\pi^2}\int_{\Gamma_{h}}\tilde v\,e^{-in\cdot\tilde x}
ds$ are the Fourier coefficients of the tangential components of $\bv|_{\Gamma_{h}}$.
Then $\Delta(e^{ik\tilde\theta\cdot\tilde x}\bv)+
k^2(e^{ik\tilde\theta\cdot\tilde x}\bv)=0$ and
$\ddiv(e^{ik\tilde\theta\cdot\tilde x}\bv)=0$ for $x_3>h$, i.e.
$\nabla\times\nabla\times(e^{ik\tilde\theta\cdot\tilde x}\bv)-
k^2(e^{ik\tilde\theta\cdot\tilde x}\bv)=0$, and $\bv$ satisfies the differential
equation of \eqref{eq:Maxwell-per}. Since $\Im\beta_n>0$ uniformly with respect
to $n$ (Lemma~\ref{lem:beta}) we conclude that $\bv$ decays exponentially as
$x_3\to\infty$. Application of Green's theorem (note that the
components on the vertical parts of the boundary cancel because of the periodicity of
$\bv$) yields
\begin{eqnarray*}
& & \int\limits_{Q_\infty}\frac{\mu_0}{\mu}\bigl[(\nabla\times\bv)\cdot(\nabla\times\overline\bpsi)
+ik\,(\nabla\times\overline\bpsi)\cdot(\check\btheta\times\bv)-
ik\,(\nabla\times\bv)\cdot(\check\btheta\times\overline\bpsi)\bigr] \\
& & -k^2\bigl( {\frac{\eps}{\eps_0}}-\frac{\mu_0}{\mu}\sin^2\theta_1\bigr)\,
\bv\cdot\overline\bpsi-k^2\frac{\mu_0}{\mu}(\check\btheta\cdot\bv)
(\check\btheta\cdot\overline\bpsi)\,dx\ =\ 0
\end{eqnarray*}
for all $\bpsi\in H_{\per,0}(\curl,Q_\infty)$. As in the scalar case this can be
written as
$$ A\bv\ -\ kB\bv\ -\ k^2C\bv\ =\ 0\,, $$
where $A$, $B$, and $C$ are self-adjoint bounded operators from
$H_{\per,0}(\curl,Q_\infty)$ into itself defined by
\begin{eqnarray}
\langle A\bv,\bpsi\rangle_{H(\curl,Q_\infty)} & = &
\int_{Q_{\infty}}\frac{\mu_0}{\mu}(\nabla\times\bv)\cdot(\nabla\times\overline{\bpsi})\,dx,
\label{eq:Op-A} \\
\langle B\bv,\bpsi\rangle_{H(\curl,Q_\infty)} & = & i\int_{Q_{\infty}}\frac{\mu_0}{\mu}
\bigl[(\nabla\times\bv)\cdot(\check\btheta\times\overline\bpsi)-
(\nabla\times\overline\bpsi)\cdot(\check\btheta\times\bv)\bigr]\,dx, \label{eq:Op-B} \\
\langle C\bv,\bpsi\rangle_{H(\curl,Q_\infty)} & = & \int_{Q_{\infty}}\bigl[
 {\frac{\eps}{\eps_0}}\,\bv\cdot\overline\bpsi-\frac{\mu_0}{\mu}
[(\check\btheta\times\bv)\times\check\btheta]\cdot
\overline\bpsi\bigr]\,dx \nonumber \\
& = & \int_{Q_{\infty}}\bigl[\bigl( {\frac{\eps}{\eps_0}}-\frac{\mu_0}{\mu}
\sin^2\theta_1\bigr)\,\bv\cdot\overline\bpsi+
\frac{\mu_0}{\mu}(\check\btheta\cdot\bv)(\check\btheta\cdot\overline\bpsi)\bigr]\,dx\,. \label{eq:Op-C}
\end{eqnarray}
Moreover, the operator $C$ is a positive operator if
 {$\eps(x)\mu(x)\geq\eps_0\mu_0\sin^2\theta_1$} in $Q_h$, and $A, B, C$ are all
self-adjoint operators. By arguing analogously to the Helmholtz case, one can show that
$\bv=0 $ if $\Im k>0$.
\end{proof}

Now we fix a real valued $k>0$ and consider the perturbation $k+i\epsilon$ with (small)
$\epsilon>0$. We write $\bff(\epsilon)=\bff_{(k+i\epsilon)}\in H_{\per,0}(\curl,Q_h)$ and
$L(\epsilon)=L_{(k+i\epsilon)}$ for $\epsilon>0$. Again, denote by
$P:H_{\per,0}(\curl,Q_h)\rightarrow\cN(L(0))$ the projection operator with respect to
the decomposition $H(\curl,Q_h)=\cN(L(0))\oplus\cR(L(0))$. From
Lemma~\ref{lem:L-prop-Maxwell} and $L(0)=J_\alpha^\ast\tilde L(0) J_\alpha$ for
$\alpha=k\tilde\theta$ it is easily seen that this decomposition is orthogonal.

\begin{lemma}\label{lem:L-deriv}
\begin{itemize}
\item[(i)] $\bff(\epsilon)\in\mathcal{R}(L(\epsilon))$ for all sufficiently small
$\epsilon>0$ and $\bff(0),\bff^\prime(0)\in\cR(L(0))$.
\item[(ii)] $PL^\prime(0)$ is one-to-one on $\cN=\cN(L(0))$ if
 {$\mu(x)\eps(x)\geq\eps_0\mu_0\sin^2\theta_1$} in $\R^3_+$.
\end{itemize}
\end{lemma}
\begin{proof}
(i) Lemma~\ref{lem:L-isomorphism} implies directly that $\bff(\epsilon)\in
\cR(L(\epsilon))$ for all sufficiently small $\epsilon>0$. We have shown
$\tilde\bff_k\in\cR(\tilde L(0))$ in part (iii) of Lemma~\ref{lem:L-prop-Maxwell} which
implies $\bff(0)\in\cR(L(0))$. In the same way we obtain from \eqref{eq:g_alpha}
that $\langle\bff^\prime(0),\bpsi\rangle_{H(\curl,Q_h)}=\bq^\prime(k)\cdot
\int_{\Gamma_h}\bnu\times\overline\bpsi\,ds=0$ and thus $\bff^\prime(0)\in\cR(L(0))$.
\smalf
(ii) The proof of this part is a bit more complicated than in the scalar case. Let
$\bv\in H_{\per,0}(\curl,Q_h)$ and $\bpsi\in\cN(L(0))$ be fixed (below we take also
$\bv\in\cN(L(0))$). First, we want to express $b_{k+i\epsilon}(\bv,\bpsi)$ by an
integral over $Q_\infty$. We set
$$ k_\epsilon:=k+i\epsilon\,,\quad \alpha_{n,\epsilon}\ :=\ k_\epsilon\tilde\theta+n
\quad\text{and}\quad\beta_{n,\epsilon}\ :=\ \sqrt{k_\epsilon^2-
\alpha_{n,\epsilon}\cdot\alpha_{n,\epsilon}} $$
for abbreviation. Let $\tilde {\cred{v}}_n\in\C^2$ be the Fourier coefficients of
$\tilde {\cred{v}}=\cred{(v^{(1)},v^{(2)})^\top}$ on $\Gamma_h$, i.e.
$\tilde {\cred{v}}(\tilde x,h) = \sum_{n\in\Z^2}\tilde v_n\,e^{in\cdot\tilde x}$. We extend
$\bv$ to $Q_\infty\setminus Q_h$ by setting $\bv_\epsilon=\bv$ in $Q_h$ and
\begin{equation}
\bv_\epsilon(x)\ =\ \sum_{n\in\Z^2}\binom{\cred{\tilde v_n}}{\cred{v}_{n,\epsilon}^{(3)}}\,
e^{in\cdot\tilde x+i\beta_{n,\epsilon}(x_3-h)}\,,\ x_3>h\,,\text{ with }
\cred{v}_{n,\epsilon}^{(3)}=-\frac{1}{\beta_{n,\epsilon}}\,
\cred{\tilde v}_n\cdot\alpha_{n,\epsilon}\,.
\end{equation}
Then it is easily seen that $\ddiv(\bv_\epsilon e^{ik_\epsilon\tilde\theta
\cdot\tilde x})=0$ and $(\Delta+k^2)(\bv_\epsilon e^{ik_\epsilon\tilde\theta
\cdot\tilde x})=0$, thus $(\nabla\times\nabla-k^2)(\bv_\epsilon
e^{ik_\epsilon\tilde\theta\cdot\tilde x})=0$, i.e. $\bv_\epsilon$ satisfies
\eqref{eq:Maxwell-per} for $k$ replaced by $k_\epsilon=k+i\epsilon$.

In the same way, but for $\epsilon=0$, we extend $\bpsi$ as a function $\bpsi_0$ to
$Q_\infty \setminus Q_h$, i.e. we set $\bpsi_0=\bpsi$ in $Q_h$ and
$$ \bpsi_0(x)\ =\ \sum_{|\alpha_n|>k}\cred{\binom{\tilde\psi_n}{\psi_n^{(3)}}}\,
e^{in\cdot\tilde x+i\beta_{n,0}(x_3-h)}\,,\ x_3>h\,,\text{ with }
\cred{\psi_n^{(3)}}=-\frac{1}{\beta_{n,0}}\,\cred{\tilde\psi_n}\cdot\alpha_{n,0}\,. $$
We note that $\bv_\epsilon$ and $\bpsi_0$ decay exponentially as $x_3\to\infty$, the
latter because $\bpsi\in\cN(L(0))$. We apply Green's theorem in
$Q_\infty\setminus Q_h$ and use the periodicity of $\bv_\epsilon$ and $\bpsi_0$
with respect to $\tilde x$. This yields
\begin{eqnarray*}
0 & = & \int\limits_{Q_\infty\setminus Q_h}
(\nabla\times\bv_\epsilon)\cdot(\nabla\times\overline{\bpsi_0})
+ik_\epsilon\,(\nabla\times\overline{\bpsi_0})\cdot(\check\btheta\times\bv_\epsilon)-
ik_\epsilon\,(\nabla\times\bv_\epsilon)\cdot(\check\btheta\times\overline{\bpsi_0}) \\
& & -k_\epsilon^2\cos^2\theta_1\,\bv_\epsilon\cdot\overline{\bpsi_0}-
k_\epsilon^2(\check\btheta\cdot\bv_\epsilon)(\check\btheta\cdot\overline{\bpsi_0})\,dx
\ - \int\limits_{\Gamma_h}\bigl[\nabla\times\bigl(\bv_\epsilon
e^{ik_\epsilon\tilde\theta\cdot\tilde x}\bigr)\bigr]\cdot(\overline{\bpsi_0}\times\bnu)\,
e^{-ik_\epsilon\tilde\theta\cdot\tilde x}ds\,.
\end{eqnarray*}
A direct computation yields
$$ e^{-ik_\epsilon\tilde\theta\cdot\tilde x}\bigl[\nabla\times
\bigl(\bv_\epsilon e^{ik_\epsilon\tilde\theta\cdot\tilde x}\bigr)\bigr]_T\ =\
\cT_{k_\epsilon}(\bnu\times\bv_\epsilon)\ =\ \cT_{k_\epsilon}(\bnu\times\bv) $$
because $\bnu\times\bv_\epsilon=\bnu\times\bv$ on $\Gamma_h$. Combining this with
the definition of $b_{k_\epsilon}$ above we have
\begin{eqnarray*}
b_{k_\epsilon}(\bv,\bpsi) & = & \int\limits_{Q_\infty}
\frac{\mu_0}{\mu}[(\nabla\times\bv_\epsilon)\cdot(\nabla\times\overline{\bpsi_0})
+ik_\epsilon\,(\nabla\times\overline{\bpsi_0})\cdot(\check\btheta\times\bv_\epsilon)-
ik_\epsilon\,(\nabla\times\bv_\epsilon)\cdot(\check\btheta\times\overline{\bpsi_0})] \\
& & \quad -\ k_\epsilon^2( {\frac{\eps}{\eps_0}}-\frac{\mu_0}{\mu}\sin^2\theta_1)\,
\bv_\epsilon\cdot\overline{\bpsi_0}-
k_\epsilon^2\frac{\mu_0}{\mu}(\check\btheta\cdot\bv_\epsilon)
(\check\btheta\cdot\overline{\bpsi_0})\,dx \\[4mm]
& = & \langle A\bv_\epsilon,\bpsi_0\rangle_{H(\curl,Q_\infty)}-
k_\epsilon\langle B\bv_\epsilon,\bpsi_0\rangle_{H(\curl,Q_\infty)}-
k_\epsilon^2\langle C\bv_\epsilon,\bpsi_0\rangle_{H(\curl,Q_\infty)}
\end{eqnarray*}
with the operators $A$, $B$ and $C$ from  \eqref{eq:Op-A},  \eqref{eq:Op-B}, and
\eqref{eq:Op-C}, respectively.
Note that $\bv_\eps\equiv\bv$ on $Q_h$. Now we differentiate this with respect to
$\epsilon$ at $\epsilon=0$. This is possible because $\bpsi_0$ decays exponentially and
$\bv_\epsilon$ is uniformly bounded as $x_3\to\infty$. Therefore, we obtain with
the derivative $\bv_0^\prime$ of $\bv_\epsilon$ at $\epsilon=0$
\begin{eqnarray*}
\frac{d}{d\epsilon}b_{k_\epsilon}(\bv,\bpsi) & = &
-i\langle B\bv_0,\bpsi_0\rangle_{H(\curl,Q_\infty)}-
2ki\langle C\bv_0,\bpsi_0\rangle_{H(\curl,Q_\infty)} \\
& & +\ \langle A\bv^\prime_0,\bpsi_0\rangle_{H(\curl,Q_\infty)}-
k\langle B\bv^\prime_0,\bpsi_0\rangle_{H(\curl,Q_\infty)}-
k^2\langle C\bv^\prime_0,\bpsi_0\rangle_{H(\curl,Q_\infty)} \\
& = & -i\langle B\bv_0,\bpsi_0\rangle_{H(\curl,Q_\infty)}-
2ki\langle C\bv_0,\bpsi_0\rangle_{H(\curl,Q_\infty)} \\
& & +\ \bigl\langle\bv^\prime_0,A\bpsi_0-kB\bpsi_0-k^2C\bpsi_0\bigr\rangle_{H(\curl,Q_\infty)}
\end{eqnarray*}
by the self-adjointness of the operators $A$, $B$, and $C$. By arguing
exactly as the proof of Lemma~\ref{lem:L-isomorphism}, we have
\begin{equation*}
A\bpsi_0 -kB\bpsi_0-k^2C\bpsi_0=0
\end{equation*}
since $\bpsi\in\cN(L(0))$, and thus
\begin{equation} \label{eq:L-deriv}
\frac{d}{d\epsilon}b_{k_\epsilon}(\bv,\bpsi)\bigr|_{\epsilon=0}\ =\
-i\,\langle B\bv_0,\bpsi_0\rangle_{H(\curl,Q_\infty)}-
2ik\,\langle C\bv_0,\bpsi_0\rangle_{H(\curl,Q_\infty)}\,.
\end{equation}
This holds for all $\bv\in H_{\per,0}(\curl,Q_h)$ and $\bpsi\in\cN(L(0))$.
Now we start with the proof of injectivity. Let $\bv\in\cN(L(0))$ with
$PL^\prime(0)\bv=0$, i.e. $\frac{d}{d\epsilon}b_k(\bv,\bpsi)=0$ for all
$\bpsi\in\cN(L(0))$, i.e. $\langle B\bv_0+2kC\bv_0,\bpsi_0\rangle_{H(\curl,Q_\infty)}=
0$ for all $\bpsi\in\cN(L(0))$ where again, $\bv_0$ and $\bpsi_0$ are the extensions of
$\bv$ and $\bpsi$, respectively. Using $A\bv_0 -kB\bv_0-k^2C\bv_0=0$ (as in the proof
of Lemma~\ref{lem:L-isomorphism}) gives
$$ \langle A\bv_0+k^2 C\bv_0,\bpsi_0\rangle_{H(\curl,Q_\infty)}\ =\ 0\quad
\text{for all }\bpsi\in\cN(L(0))\,. $$
We take $\bpsi=\bv$ and note that $A$ and $C$ are non-negative, $C$ even positive.
This yields $\bv_0=0$ and ends the proof.
\end{proof}	
Let $\bF(\epsilon)\in H_{\per,0}(\curl,Q_h)$ be the unique solution of
$$ L(\epsilon)\;\bF(\epsilon)\ =\ \bff(\epsilon)\,,\quad\text{for }\epsilon>0. $$
Lemmas \ref{lem:L-isomorphism}-\ref{lem:L-deriv} shows that the singular perturbation
arguments of Lemma \ref{lem:sing-pert} is applicable to the operator $L(\epsilon)$.
Consequently, $\bF(\epsilon)$ converges to $\bF:=\bF(0)$ in $H_{\per,0}(\curl,Q_h)$
and the limiting function $\bF\in H_{\per,0}(\curl,Q_h)$ satisfies the equations
\begin{equation*}
L_k\bF\ =\ \bff_k\quad\mbox{and}\quad PL^\prime(0)\bF\ =\ 0\,.
\end{equation*}
As in the Helmholtz case, the first equation asserts that the extension of the function
$\bE(x)=\bF(x)^{ik\tilde\theta\cdot\tilde x}$ solves the diffraction problem
\eqref{eq:odp-2}-\eqref{eq:Rayleigh-Maxwell} with the real-valued wavenumber $k>0$, while
the second equation provides an additional constraint on $\bE$ to ensure uniqueness when
$\alpha=k\tilde\theta$ is a  {propagative wave vector}.
The equation $PL^\prime(0)\bF=0$ is equivalent to
$\frac{d}{d\epsilon}b_{k_\epsilon}(\bF,\bpsi)\bigr|_{\epsilon=0}=0$ for all
$\bpsi\in\cN(L(0))$, i.e. by \eqref{eq:L-deriv},
$$ \langle B\bF_0,\bpsi_0\rangle_{H(\curl,Q_\infty)}+
2k\,\langle C\bF_0,\bpsi_0\rangle_{H(\curl,Q_\infty)}\ =\ 0\quad\text{for all }
\bpsi\in\cN(L(0))\,. $$
Note that $\bF_0$ and $\bpsi_0$ are the extensions of $\bF$ and $\bpsi$, respectively,
into $Q_\infty$. Using the definitions of $B$ and $C$ from \eqref{eq:Op-B},
\eqref{eq:Op-C} we obtain
\begin{eqnarray*}
& & i\int_{Q_{\infty}}\frac{\mu_0}{\mu}\bigl[(\nabla\times\bF_0)\cdot
(\check\btheta\times\overline\bpsi_0)-
(\nabla\times\overline\bpsi_0)\cdot(\check\btheta\times\bF_0)\bigr]\,dx \\
& & +\ 2k \int_{Q_{\infty}}\bigl[ {\frac{\eps}{\eps_0}}\,\bF_0\cdot\overline\bpsi_0-
\frac{\mu_0}{\mu}[(\check\btheta\times\bF_0)\times\check\btheta]\cdot\overline\bpsi_0\bigr]
\ =\ 0
\end{eqnarray*}
In the quasi-periodic setting, i.e. for $\bE=e^{ik\tilde\theta\cdot\tilde x}\bF_0$, this
equation is equivalent to
\begin{eqnarray}\label{orMax}
2k\int_{Q_\infty} {\frac{\eps}{\eps_0}}\,\bE\cdot\overline\bpsi\,dx & = &
i\int_{Q_{\infty}}\frac{\mu_0}{\mu}\bigl[(\nabla\times\overline\bpsi)\cdot
(\check\btheta\times\bE)
-(\nabla\times\bE)\cdot(\check\btheta\times\overline\bpsi)\bigr]\,dx \label{con-max} \\
&=&-i\,\check{\btheta}\cdot \int_{Q_{\infty}}\frac{\mu_0}{\mu}\left[
\overline\bpsi\times\left(\nabla\times\bE\right)-\bE\times
 \left(\nabla\times\overline{\boldsymbol{\psi}}\right)\right]dx \nonumber
\end{eqnarray}
for all modes $\bpsi\in\cM_k$. Multiplying $k$ to both sides of \eqref{orMax} and using $\hat{\balpha}=k\check{\btheta}=(\alpha,0)$, we can rewrite the previous identity as
\ben
2k^2\int_{Q_\infty} {\frac{\eps}{\eps_0}}\,\bE\cdot\overline\bpsi\,dx & = &
i\int_{Q_{\infty}}\frac{\mu_0}{\mu}\bigl[(\nabla\times\overline\bpsi)\cdot
(\hat\balpha\times\bE)
-(\nabla\times\bE)\cdot(\hat\balpha\times\overline\bpsi)\bigr]\,dx \label{con-max} \\
&=&-i\,\hat{\balpha}\cdot \int_{Q_{\infty}}\frac{\mu_0}{\mu}\left[
\overline\bpsi\times\left(\nabla\times\bE\right)-\bE\times
 \left(\nabla\times\overline{\boldsymbol{\psi}}\right)\right]dx. 
\enn

\smalf

We summarise uniqueness and existence of the electromagnetic diffraction problem below.
\begin{theorem}\label{TH-max} Let $\alpha=k\tilde\theta=
k\sin\theta_1(\cos\theta_2,\sin\theta_2)^\top\in \R^2$.
\begin{itemize}
\item[(i)] If $\alpha$ is not a  {propagative wave vector}, then
there exists a unique solution $\bE\in H_{\alpha,\loc,0}(\curl,Q_\infty)$ such that
$\bE^{sc}:=\bE-\bE^{in}$ satisfies the $\alpha$-quasiperiodic Rayleigh expansion
\eqref{eq:Rayleigh-Maxwell}.
\item[(ii)] Suppose that $\alpha$ is a  {propagative wave vector,
$\eps(x)\mu(x)\geq\eps_0\mu_0\sin^2\theta_1$} in $\R^3_+$ and that $|\alpha+n|\neq k$
for all $n\in \Z^2$. Then the diffraction problem still admits a unique solution
$\bE\in H_{\alpha,\loc,0}(\curl,Q_\infty)$ provided $\bE$ additionally satisfies
the constraint condition \eqref{con-max}.
\end{itemize}
\end{theorem}

\begin{remark} (i)
If $k\in \R$ and  {$\eps(x)\mu(x)\geq\eps_0\mu_0\sin^2\theta_1$ on $Q_{h}$} is not
satisfied, one can construct examples as in the scalar case which show that guided/surface waves for the Maxwell equations do
not always exist.
\cred{ Let $h=1$ and suppose that $\epsilon(x)\equiv \epsilon_1, \mu(x)\equiv \mu_1$ in $0<x_3<1$. Assume that $q_0=\epsilon_0 \mu_0/(\epsilon_1\mu_1)<1$. We expand the guided wave mode into the series
\ben
&&\bE(x)=\sum_{|\alpha_n|>k} \begin{pmatrix}
	a_n^+ \\ b_n^+ \\ c_n^+
\end{pmatrix} e^{i\alpha_n \cdot \tilde x}\,e^{-|\beta_n| (x_3-1)}\quad\mbox{in}\quad x_3>1,\quad \alpha_n=(\alpha_n^{(1)}, \alpha_n^{(2)}),\\
&&\bE(x)=\sum_{|\alpha_n|>k} \begin{pmatrix}
	a_n^- \\ b_n^- \\ c_n^-
\end{pmatrix} e^{i\alpha_n \cdot \tilde x}\,e^{-|\gamma_n| x_3}+
\begin{pmatrix}
	-a_n^- \\ -b_n^- \\ c_n^-
\end{pmatrix} e^{i\alpha_n \cdot \tilde x}\,e^{|\gamma_n| x_3}
\quad\mbox{in}\quad 0<x_3<1,
\enn
where $\gamma_n:=\sqrt{k^2q_0-|\alpha_n|^2}=i|\gamma_n|$ if $|\alpha_n|>k$. Note that the form of $\bE$ ensures $\bnu\times \bE=0$ on $x_3=0$. The divergence-free condition $\nabla\cdot\bE=0$ in $x_3>h$ and $0<x_3<h$ leads to two equations
\begin{equation}\label{E1}
(a_n^+, b_n^+)\cdot i \alpha_n - c_n^+ |\beta_n|=0,\quad
(a_n^-, b_n^-)\cdot i \alpha_n - c_n^- |\gamma_n|=0.
\end{equation}
Meanwhile, we get from the continuity of $\bnu\times \bE$ and  $\bnu\times(\nabla\times\bE)$ on $x_3=1$ that
\be\label{E2}
\begin{pmatrix}
	a_n^+ \\ b_n^+
\end{pmatrix}=
\begin{pmatrix}
	a_n^- \\ b_n^-
\end{pmatrix} \left(e^{-|\gamma_n|}- e^{|\gamma_n|}\right)
\en
and
\be\label{E4}\left\{\begin{array}{lll}
b_n^+ |\beta_n|+i c_n^+ \alpha_n^{(2)}&=&(b_n^- |\gamma_n|+i c_n^- \alpha_n^{(2)})  \left(e^{-|\gamma_n|}+ e^{|\gamma_n|}\right),\\
i\,c_n^+\alpha_n^{(1)}+a_n^+|\beta_n|&=&(i\,c_n^-\alpha_n^{(1)}+a_n^-|\gamma_n|) \left(e^{-|\gamma_n|}+ e^{|\gamma_n|}\right).
\end{array}\right.
\en
The relations in \eqref{E1} and \eqref{E2} imply that
\begin{equation}\label{E3}
	c_n^+ |\beta_n|= c_n^- |\gamma_n| \,\left(e^{-|\gamma_n|}- e^{|\gamma_n|}\right).
\end{equation}
%Case (i): $a_n^\pm\neq 0, b_n^\pm\neq 0$.  From \eqref{E4}, it follows that
%\ben\left\{\begin{array}{lll}
% b_n^+ |\beta_n|+i c_n^+ \alpha_n^{(2)}&=&b_n^- |\gamma_n|+i c_n^- \alpha_n^{(2)},\\
%i\,c_n^+\alpha_n^{(1)}+a_n^+|\beta_n|&=&i\,c_n^-\alpha_n^{(1)}+a_n^-|\gamma_n|.
%\end{array}\right.
%\enn
Using \eqref{E1}, \eqref{E2} and \eqref{E3}, we may rewrite the  system \eqref{E4} as the algebraic equations for $(a_n^+, b_n^+, c_n^+)$:
\ben
\begin{pmatrix}
h_n & 0 & i\alpha_n^{(1)} d_n\\
0 & h_n & i\alpha_n^{(2)} d_n \\	
i	\alpha_n^{(1)} &
i	\alpha_n^{(2)} &
-|\beta_n|
\end{pmatrix}
\begin{pmatrix}
	a_n^+\\ b_n^+\\ c_n^+
\end{pmatrix}=0,
\enn
with
\ben
h_n:=|\beta_n|-|\gamma_n|\,\frac{ e^{-|\gamma_n|}+e^{|\gamma_n|}}{ e^{-|\gamma_n|}-e^{|\gamma_n|}},
\quad d_n:=1-\frac{|\beta_n|}{|\gamma_n|}\frac{ e^{-|\gamma_n|}+e^{|\gamma_n|}}{e^{-|\gamma_n|}-e^{|\gamma_n|}}. \enn
Straightforward calculations show the determinant of the coefficient matrix given by
\ben
k^2\left[1-q_0\frac{|\beta_n|}{|\gamma_n|}
\frac{e^{-|\gamma_n|}+e^{|\gamma_n|}}{ e^{-|\gamma_n|}-e^{|\gamma_n|}}  \right]>0.
\enn
This proves $a_n^\pm=b_n^\pm=c_n^\pm=0$ and thus $u=0$.
}
%\tcr{On the other hand, even for $\Im k>0$ it remains unclear to prove
%uniqueness of solutions if the lower bound condition  {$\eps(x)\mu(x)\geq\eps_0\mu_0
%\sin^2\theta_1$} on $Q_{h}$ is not satisfied.}
\end{remark}

\section{Appendix}
In this section, we present the singular perturbation argument used in the proofs of
Theorem \ref{TH-LAP} (ii)  and Theorem \ref{TH-max} (ii), and prove a compactness
result.

\subsection{A singular perturbation result}

\begin{lemma}[see \cite{K22p}]\label{lem:sing-pert}
Let $I=(0, \epsilon_0)$ for some small $\epsilon_0>0$. Let $L(\epsilon)$ be Fredholm
operators from some Hilbert space $X$ into itself and $f(\epsilon)\in \cR(L(\epsilon))$
for all $\epsilon\in[0, \epsilon_0)$ where $\cR(L)$ denotes the range of an operator $L$.
Furthermore, let $L(\epsilon)$ be one-to-one (thus invertible) for all $\epsilon\in I$
and let $L(0)$ have Riesz number
one, i.e. $\cN(L(0)^2)=\cN(L(0))$ where $\cN(L)$ denotes the nullspace of an operator $L$.
Let  $P:X\rightarrow\cN:=\cN(L(0))$ be the projection onto the nullspace of $L(0)$
along the direct decomposition $X=\cN\oplus\cR(L(0))$.  Finally, let $\epsilon\mapsto
f(\epsilon)$ and $\epsilon\mapsto L(\epsilon)$ be continuously differentiable for
$\epsilon\in [0,\epsilon_0)$ and let $PL^\prime(0)|_{\mathcal{N}}$ be an isomorphism from
$\cN$ onto itself where $L^\prime(0)$ denotes the (one-sided) derivative of $L(\epsilon)$
at $\epsilon=0$.

Then the mapping $\epsilon\mapsto v(\epsilon):=[L(\epsilon)]^{-1}f(\epsilon)$ has a
continuous  extension to a mapping from $[0,\epsilon_0)$ into $X$. The limit
$v(0)=\lim_{\epsilon\rightarrow 0+}v(\epsilon)$ is the unique solution of the system
\be\label{elimit}
L(0)\,v(0)\ =\ f(0)\,,\quad [PL^\prime(0)]\,v(0)\ =\ Pf^\prime(0)\,,
\en
where $f^\prime(0)$ denotes the one-sided derivative of $f(\epsilon)$ at $\epsilon=0$.
Moreover, there exist $\delta\in(0,\epsilon_0)$ and $c>0$ such that
\ben
\sup_{\epsilon\in[0,\delta]}\Vert v(\epsilon)\Vert_X\ \leq\
c\,\bigl[\sup_{\epsilon\in[0,\delta]}
\Vert f(\epsilon)\Vert_X+\sup_{\epsilon\in[0,\delta]}
\Vert f^\prime(\epsilon)\Vert_X\bigr]\,.
\enn
\end{lemma}
The original version of Lemma \ref{lem:sing-pert} can be found in \cite[Theorem 1.32,
Section 1.4]{CK13}. A more direct proof is presented in \cite{KS24} and \cite{K22p} with
the characterization of the equation \eqref{elimit} of the limiting solution.

\subsection{Compact embedding theorem of $X_{\alpha,0}$}

\begin{lemma} \label{lem:Hodge}
The space
$$ X_{\alpha,0}\ =\ \left\{\bu\in H_{\alpha,0}(\curl,Q_h):\tilde b_\alpha(\bu,\nabla p)=0
\mbox{ for all } p \in H^1_{\alpha,0}(Q_h)\right\} $$
is compactly embedded in $L^2(Q_h)^{3}$.
\end{lemma}
\begin{proof}
From the definition of the sesquilinear form $\tilde b_\alpha$,
$$ \tilde b_\alpha(\bu,\bpsi)\ =\ \int_{Q_h}\bigl[(\nabla\times\bu)\cdot
(\nabla\times\overline{\bpsi})-k^2q\,\bu\cdot\overline{\bpsi}\bigr]\,dx-\int_{\Gamma_h}
\widetilde\cT_\alpha(\bnu\times \bu)\cdot(\bnu\times\overline{\bpsi})\,ds\,, $$
it follows for $\bu\in X_{\alpha,0}$ and $\bpsi=\nabla p$ for $p\in H^1_{\alpha,0}(Q_h)$
that $$-k^2\int_{Q_h}q\,\bu\cdot\nabla\overline p\,dx\ =\ \int_{\Gamma_h}
\widetilde\cT_\alpha(\bnu\times \bu)\cdot(\bnu\times\nabla\overline  p)\,ds $$
Applying the divergence theorem in $Q_h$ on the left hand side and on $\Gamma_h$ on the
right hand side (i.e. $\int_{\Gamma_h}\bv\cdot\nabla p\,ds=-\int_{\Gamma_h}\Div\bv\,p\,ds$
where $\Div$ denotes the surface divergence on $\Gamma_h$) we get
$$ k^2 \int_{Q_h}\ddiv (q\,\bu)\,\overline{p}\,dx\ +\ \int_{\Gamma_{h}}\left\{
\Div\bigl[\widetilde\cT_\alpha(\bu\times\bnu)\times\bnu\bigr]-\bu\cdot\bnu\right\}\,
\overline{p}\,ds\ =\ 0\,. $$
Since $p$ is arbitrary we obtain
\begin{align*}
\ddiv(q\bu)& = 0 \quad \mbox{in }Q_h, \\
\bu\cdot\bnu & = \frac{1}{k^2}\Div\bigl[\widetilde\cT_\alpha(\bu\times\bnu)\times\bnu\bigr]
\quad\mbox{on }\Gamma_h.
\end{align*}
Applying the compactness result of \cite{Weber}, one can show that
the space $X_{\alpha,0}$ is compactly embedded in $L^2(Q_h)^{3}$; see
\cite[Lemma 3.26]{GP22}.
\end{proof}

\end{document}